\documentclass[a4paper,10pt]{amsart}%
\usepackage{amsfonts}
\usepackage{amssymb,latexsym}
\usepackage{amsmath}
\usepackage{amssymb}
\usepackage{graphicx}%
\theoremstyle{plain}
\newtheorem{theorem}{Theorem}
\newtheorem{corollary}[theorem]{Corollary}
\newtheorem{lemma}[theorem]{Lemma}
\newtheorem{proposition}[theorem]{Proposition}

\theoremstyle{definition}

\theoremstyle{remark}
\newtheorem{example}[theorem]{Example}

\newtheorem{remark}[theorem]{Remark}
\numberwithin{equation}{section}
\numberwithin{theorem}{section}
 	
\begin{document}
\title{On composition-closed classes of Boolean functions}
\author{Tam\'{a}s Waldhauser}
\address{Mathematics Research Unit\\
University of Luxembourg\\
6 rue Richard Coudenhove-Kalergi\\
L-1359 Luxembourg, Luxembourg, and\\
Bolyai Institute\\
University of Szeged\\
Aradi v\'{e}rtan\'{u}k tere 1\\
H-6720 Szeged, Hungary}
\email{twaldha@math.u-szeged.hu}

\begin{abstract}
We determine all composition-closed equational classes of Boolean functions.
These classes provide a natural generalization of clones and iterative
algebras: they are closed under composition, permutation and identification
(diagonalization) of variables and under introduction of inessential variables
(cylindrification), but they do not necessarily contain projections. Thus the
lattice formed by these classes is an extension of the Post lattice. The
cardinality of this lattice is continuum, yet it is possible to describe its
structure to some extent.

\end{abstract}
\maketitle

\section{Introduction\label{sect intro}}

The goal of this paper is to describe composition-closed equational classes of
Boolean functions not necessarily containing projections, thereby generalizing
Post's description of Boolean clones. First we recall the definition of a
clone, and then we give an informal overview of the problem that we consider.
For formal definitions and more background see
Section~\ref{sect preliminaries} and \cite{Lau,PK}.

We define the \emph{composition} of an $n$-ary function $f\colon
A^{n}\rightarrow A$ by the $k$-ary functions $g_{1},\ldots,g_{n}\colon
A^{k}\rightarrow A$ as the $k$-ary function $f\left(  g_{1},\ldots
,g_{n}\right)  $ given by%
\begin{equation}
f\left(  g_{1},\ldots,g_{n}\right)  \left(  \mathbf{a}\right)  =f\left(
g_{1}\left(  \mathbf{a}\right)  ,\ldots,g_{n}\left(  \mathbf{a}\right)
\right)  \text{ for all }\mathbf{a}\in A^{k}\text{.} \label{eq comp}%
\end{equation}
We say that $f$ is the \emph{outer function} of the composition, and
$g_{1},\ldots,g_{n}$ are the \emph{inner functions}. A \emph{clone} on the set
$A$ is a class $\mathcal{C}\subseteq\bigcup_{n\geq1}A^{A^{n}}$ of finitary
functions that is closed under composition and contains the \emph{projections}%
\[
e_{i}^{\left(  n\right)  }\colon A^{n}\rightarrow A,\left(  x_{1},\ldots
,x_{n}\right)  \mapsto x_{i}\quad\left(  n\in\mathbb{N}\hspace{0cm},1\leq
i\leq n\right)  .
\]
(Here, and in the rest of the paper, $\mathbb{N}=\left\{  1,2,\ldots\right\}
$, i.e., we exclude $0$ from the set of natural numbers.)

Although the above definition of composition is restrictive in the sense that
the inner functions must have the same arity, by making use of projections one
can see that clones are closed under compositions without restrictions on the
arities. For example, let us suppose that $f$ is a ternary function in a clone
$\mathcal{C}$, and $g_{1},g_{2},g_{3}$ are unary, binary, ternary functions in
$\mathcal{C}$, respectively. If we would like to build the composite function
$h\left(  x_{1},x_{2},x_{3}\right)  =f\left(  g_{1}\left(  x_{1}\right)
,g_{2}\left(  x_{2},x_{1}\right)  ,g_{3}\left(  x_{1},x_{1},x_{3}\right)
\right)  $ using only compositions of the form (\ref{eq comp}), then we could
proceed as follows: first construct ternary functions $g_{1}^{\prime}%
,g_{2}^{\prime},g_{3}^{\prime}\in\mathcal{C}$ with the help of the
projections:%
\begin{align*}
g_{1}^{\prime}  &  =g_{1}(e_{1}^{\left(  3\right)  }),\\
g_{2}^{\prime}  &  =g_{2}(e_{2}^{\left(  3\right)  },e_{1}^{\left(  3\right)
}),\\
g_{3}^{\prime}  &  =g_{3}(e_{1}^{\left(  3\right)  },e_{1}^{\left(  3\right)
},e_{3}^{\left(  3\right)  }),
\end{align*}
and then form the composition $h=f\left(  g_{1}^{\prime},g_{2}^{\prime}%
,g_{3}^{\prime}\right)  \in\mathcal{C}$.

As we can see from the above example, composing a function with projections
allows us to add dummy variables to the function (see $g_{1}^{\prime}$), to
permute the variables of the function (see $g_{2}^{\prime}$) and to identify
variables of the function (see $g_{3}^{\prime}$). Function classes closed
under the latter three operations are called \emph{equational classes}, since
they can be defined by functional equations (see
Subsection~\ref{subsect equational classes}). The above discussion shows that
every clone is an equational class, but, as we shall see in
Subsection~\ref{subsect equational classes}, there are equational classes that
are not clones.

All clones on a given finite base set $A$ form an algebraic lattice. This
\emph{clone lattice} has continuum cardinality if $\left\vert A\right\vert
\geq3$ (see \cite{JM}), and it seems to be a very hard problem to describe its
structure. The case $\left\vert A\right\vert =2$, i.e., the case of Boolean
functions was settled by E.~L.~Post, who described all clones of Boolean
functions in \cite{Post}.\ There are countably many such clones, and their
lattice is known as the Post lattice (see Figure~\ref{fig post}).

We will generalize the notion of a clone by considering function classes that
are closed under composition (in the sense of (\ref{eq comp})) but do not
necessarily contain the projections. However, we would like to be able to
identify and permute variables and introduce dummy variables, therefore we
only consider composition-closed equational classes. These classes form a
complete lattice that contains the clone lattice as the principal filter
generated by the clone of projections (see Figure~\ref{fig biglattice}). The
main result of this paper is a description of this lattice over a two-element
base set. Although the clone lattice is countable in this case, we will see
that the lattice of composition-closed equational classes of Boolean functions
is uncountable.

Composition-closed equational classes subsume iterative algebras as well. A
function class $\mathcal{K}$ is an \emph{iterative algebra} if $f\left(
g_{1},\ldots,g_{n}\right)  \in\mathcal{K}$, whenever $f\in\mathcal{K}$ and
$g_{i}\in\mathcal{K}\cup\left\{  \text{projections}\right\}  \mathcal{\ }$for
$i=1,2,\ldots,n$. Clearly, every iterative algebra is a composition-closed
equational class, and an iterative algebra is a clone iff it contains the projections.

The difference between composition-closed equational classes and iterative
algebras can be best understood by visualizing compositions as trees. As an
example, let us consider the following two compositions:

\begin{center}%
\raisebox{-0cm}{\includegraphics[
height=4.6568cm,
width=12.8372cm
]%
{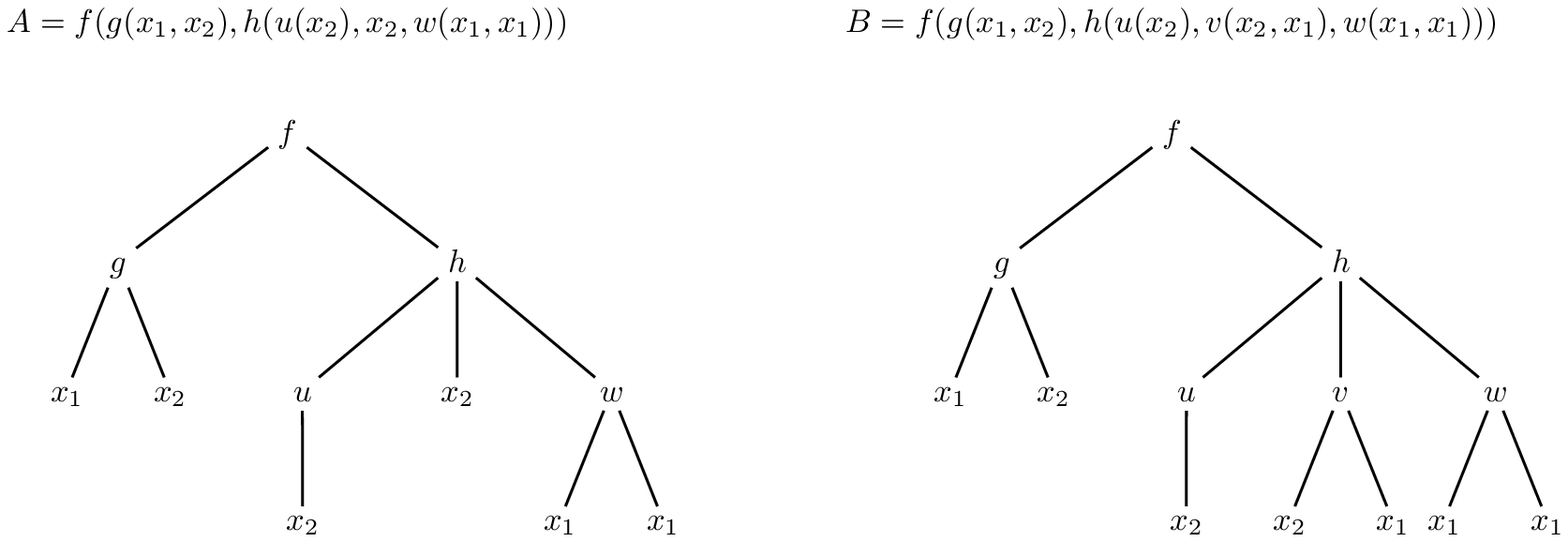}%
}%

\end{center}

\noindent If $\mathcal{K}$ is an iterative algebra and $f,g,h,u,v,w\in
\mathcal{K}$, then both $A$ and $B$ must belong to $\mathcal{K}$. However, if
$\mathcal{K}$ is only assumed to be a composition-closed equational class,
then $A$ does not necessarily belong to $\mathcal{K}$, as it involves the
composition $h\left(  u\left(  x_{2}\right)  ,x_{2},w\left(  x_{1}%
,x_{1}\right)  \right)  $, where one inner function is a projection. This
problem does not arise with $B$, hence $B\in\mathcal{K}$ is guaranteed,
whenever $\mathcal{K}$ is a composition-closed equational class.

In general, we can say that composition-closed equational classes are closed
under compositions whose tree satisfies the following condition: for any
internal node, either all or none of its children are leaves. As an exercise
in handling such compositions, we invite the reader to verify the following
fact: The clone generated by the addition operation of a field (or, more
generally, of any additive commutative semigroup) consists of functions of the
form $\sum a_{i}x_{i}$~$\left(  a_{i}\in\mathbb{N}\right)  $. The iterative
algebra generated by addition contains only those such functions where $\sum
a_{i}\geq2$, while the composition-closed equational class generated by
addition contains only those where $\sum a_{i}$ is even.

The paper is organized as follows: In Section \ref{sect preliminaries} we
present the necessary background on equational classes, Boolean clones, and
Galois connections between functions and relations. In Section
\ref{sect idempotents vs clones} we make some basic observations about
composition-closed equational classes of Boolean functions, and we outline a
strategy for constructing all of them. In Section \ref{sect easy} we carry out
this strategy for the easy cases, and then we deal with the harder cases in
Sections \ref{sect W_infty} and \ref{sect W_k}. Finally, in Section
\ref{sect concluding} we put together all the information we found to get a
picture about the lattice of composition-closed equational classes of Boolean functions.

\section{Preliminaries\label{sect preliminaries}}

\subsection{Subfunctions\label{subsect subfunctions}}

Let $f$ and $g$ be operations on a set $A$ of arity $n$ and $m$, respectively.
If there exists a map $\sigma\colon\left\{  1,2,\ldots,n\right\}
\rightarrow\left\{  1,2,\ldots,m\right\}  $ such that%
\[
g\left(  x_{1},\ldots,x_{m}\right)  =f\left(  x_{\sigma\left(  1\right)
},\ldots,x_{\sigma\left(  n\right)  }\right)  ,
\]
then we say that $g$ is a \emph{subfunction} (or identification minor or
simple variable substitution) of $f$, and we denote this fact by $g\preceq f$.
If $\sigma$ is bijective, then $g$ is obtained from $f$ by permuting
variables; if $\sigma$ is not injective, then $g$ is obtained from $f$ by
identifying variables; if $\sigma$ is not surjective, then $g$ is obtained
from $f$ by introducing inessential (dummy) variables.

The subfunction relation gives rise to a quasiorder on the set of all finitary
functions on $A$ (see \cite{quasi-ordering}). The corresponding
\emph{equivalence} is defined by $f\equiv g\iff f\preceq g$ and $g\preceq f$,
and it is clear that $f$ and $g$ are equivalent iff they differ only in
inessential variables and/or in the order of their variables. We will not
distinguish between equivalent functions in the sequel. For example, we will
denote the set of all constant zero functions (for any $0\in A$) simply by
$\left\{  0\right\}  $, and $\left\{  \operatorname{id}\right\}  $ will stand
for the set of all projections, as these are the functions equivalent to the
identity function.

Let $\Omega$ denote the class of all\ Boolean functions, i.e., the class of
finitary operations on $A=\left\{  0,1\right\}  $. The subfunction relation
induces naturally a partial order on $\Omega/\equiv$; the bottom of this poset
is shown in Figure~\ref{fig subf}. We can see (and it is easy to prove) that
it has four connected components, namely $\Omega_{00},\Omega_{11},\Omega
_{01},\Omega_{10}$, where%
\[
\Omega_{ab}=\left\{  f\in\Omega:f\left(  \mathbf{0}\right)  =a,\,f\left(
\mathbf{1}\right)  =b\right\}  \quad\left(  a,b\in\left\{  0,1\right\}
\right)  .
\]
Let us observe that $\Omega_{01}$ is nothing else but the clone of idempotent
functions. For an arbitrary function class $\mathcal{K}$, we will abbreviate
$\mathcal{K}\cap\Omega_{ab}$ by $\mathcal{K}_{ab}$, and we will later use the
following (hopefully intuitive) notation as well:%
\[
\mathcal{K}_{0\ast}=\mathcal{K}_{00}\cup\mathcal{K}_{01},~\mathcal{K}_{\ast
1}=\mathcal{K}_{01}\cup\mathcal{K}_{11},~\mathcal{K}_{=}=\mathcal{K}_{00}%
\cup\mathcal{K}_{11}.
\]

\begin{figure}
[ptb]
\begin{center}
\includegraphics[
height=3.082cm,
width=12.6322cm
]%
{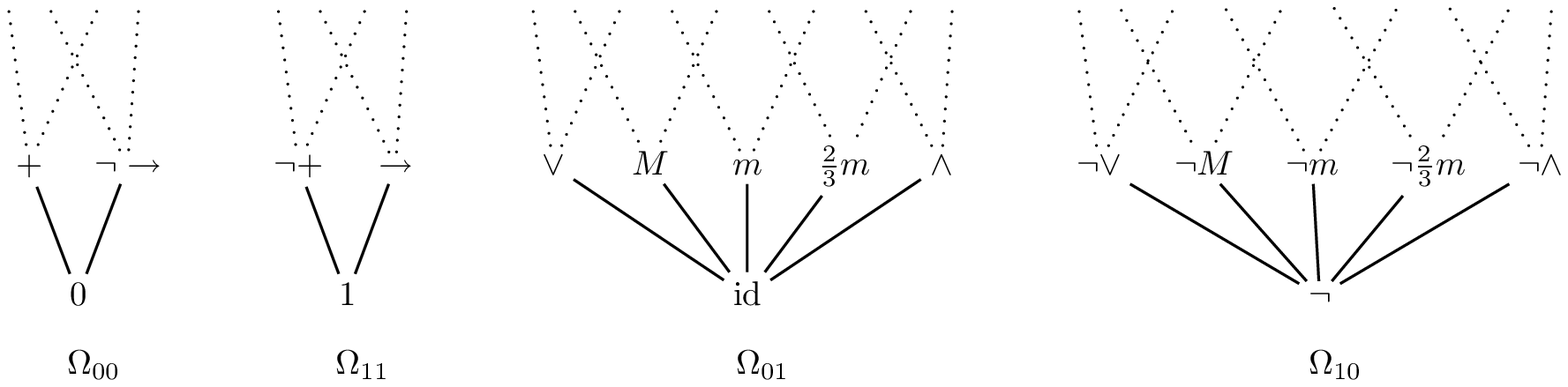}%
\caption{The subfunction quasiorder on Boolean functions}%
\label{fig subf}%
\end{center}
\end{figure}
The minimal elements of $\left(  \Omega/\equiv;\preceq\right)  $ are the unary
functions: $0,1,\operatorname{id}$ and $\lnot$ (negation). On the next level
we can see the binary operations $+$ (addition modulo $2$), $\rightarrow$
(implication), $\vee$ (disjunction), $\wedge$ (conjunction) and the ternary
functions $M$ (majority operation), $m$ (minority operation), $\frac{2}{3}m$
($\frac{2}{3}$-minority operation, see \cite{BurSan}) together with their
negations. Here negation is \textquotedblleft taken from
outside\textquotedblright, e.g., $\lnot\frac{2}{3}m$ is a shorthand notation
for the function $\lnot\frac{2}{3}m\left(  x,y,z\right)  =1+\frac{2}%
{3}m\left(  x,y,z\right)  =1+xy+yz+xz+x+z$.

\subsection{Equational classes\label{subsect equational classes}}

A class $\mathcal{K}$ of operations on a set $A$ is an \emph{equational class}
if it is an order ideal in the subfunction quasiorder, i.e., if $f\in
\mathcal{K}$ and $g\preceq f$ imply $g\in\mathcal{K}$. The denomination is
explained by the fact that these are exactly the classes that can be defined
by functional equations \cite{CF,EFHH}. Clearly every clone is an equational
class, but there are other equational classes as well; a natural example is
the class of antimonotone (order reversing) Boolean functions, which can be
defined by the functional equation $f\left(  \mathbf{x}\wedge\mathbf{y}%
\right)  \wedge f\left(  \mathbf{x}\right)  =f\left(  \mathbf{x}\right)  $.
Another example is the class $\Omega_{=}$ (see
Subsection~\ref{subsect subfunctions}), which can be defined by the equation
$f\left(  \mathbf{0}\right)  =f\left(  \mathbf{1}\right)  $. Equational
classes can be characterized by relational constraints as well; we will
discuss this in more detail in Subsection~\ref{subsect relational constarints}.

The equational classes on a given set $A$ form a lattice $\mathbf{E}_{A}$ with
intersection and union as the lattice operations. This lattice has continuum
cardinality already on the two-element set, and its structure is very
complicated \cite{quasi-ordering}. If $\mathcal{A}$ and $\mathcal{B}$ are
classes of functions, then their \emph{composition}, denoted by $\mathcal{A}%
\circ\mathcal{B}$, is the set of all compositions where the outer function
belongs to $\mathcal{A}$ and the inner functions belong to $\mathcal{B}$:%
\[
\mathcal{A}\circ\mathcal{B}=\left\{  f\left(  g_{1},\ldots,g_{n}\right)
:f\in\mathcal{A},~g_{1},\ldots,g_{n}\in\mathcal{B}\right\}  .
\]
In general, associativity does not hold for function class composition, but it
holds for equational classes \cite{C,composition of post classes}, hence we
obtain a monoid $\left(  \mathbf{E}_{A};\circ\right)  $ with the identity
element $\left\{  \operatorname{id}\right\}  $.

We will mostly consider Boolean functions, and in this case we will drop the
index $A$, and denote the set of equational classes simply by $\mathbf{E}$. A
class $\mathcal{K}$ is \emph{closed under composition} iff $\mathcal{K}%
\circ\mathcal{K}\subseteq\mathcal{K}$. As we shall see in Proposition
\ref{prop (sub)idempotents}, this is equivalent to the formally stronger
requirement that $\mathcal{K}$ is \emph{idempotent}, i.e., $\mathcal{K}%
\circ\mathcal{K}=\mathcal{K}$ (cf.~\cite{C}). (Let us note that this is a
distinguishing feature of Boolean functions: if $A$ has at least three
elements, then one can construct a class $\mathcal{K}\in\mathbf{E}_{A}$ such
that $\mathcal{K}\circ\mathcal{K}\subsetneq\mathcal{K}$.) The goal of this
paper is to describe the idempotent elements of $\left(  \mathbf{E}%
;\circ\right)  $.

The usual notation for the set of idempotents of a semigroup $\mathbf{S}$ is
$E\left(  \mathbf{S}\right)  $, but in our case this would lead to the
somewhat awkward notation $E\left(  \mathbf{E}\right)  $, therefore we will
simply write $\mathbf{I}$ for the set of composition-closed equational classes
over $\left\{  0,1\right\}  $. Clearly $\mathbf{I}$ is closed under arbitrary
intersections (we allow the empty class), hence it is a complete lattice. The
lattice of clones appears in $\mathbf{I}$ as the principal filter generated by
$\left\{  \operatorname{id}\right\}  $, and we will see that the rest of
$\mathbf{I}$ is the principal ideal generated by $\Omega_{=}$ (see
Figure~\ref{fig biglattice}). We will also see that the lattice $\mathbf{I}$
has continuum cardinality.

\subsection{The Post lattice\label{subsect Post lattice}}

There are countably many clones on the two-element set, and these have been
described by E.~L.~Post in \cite{Post}. Figure~\ref{fig post} shows the clone
lattice on $\left\{  0,1\right\}  $, usually referred to as the \emph{Post
lattice}. The top element is $\Omega$, the class of all Boolean functions, and
the bottom element is $\left\{  \operatorname{id}\right\}  $, the clone
consisting of projections only. The other clones labelled in the figure are
the following:%
\begin{figure}
[ptb]
\begin{center}
\includegraphics[
height=11.5666cm,
width=12.6322cm
]%
{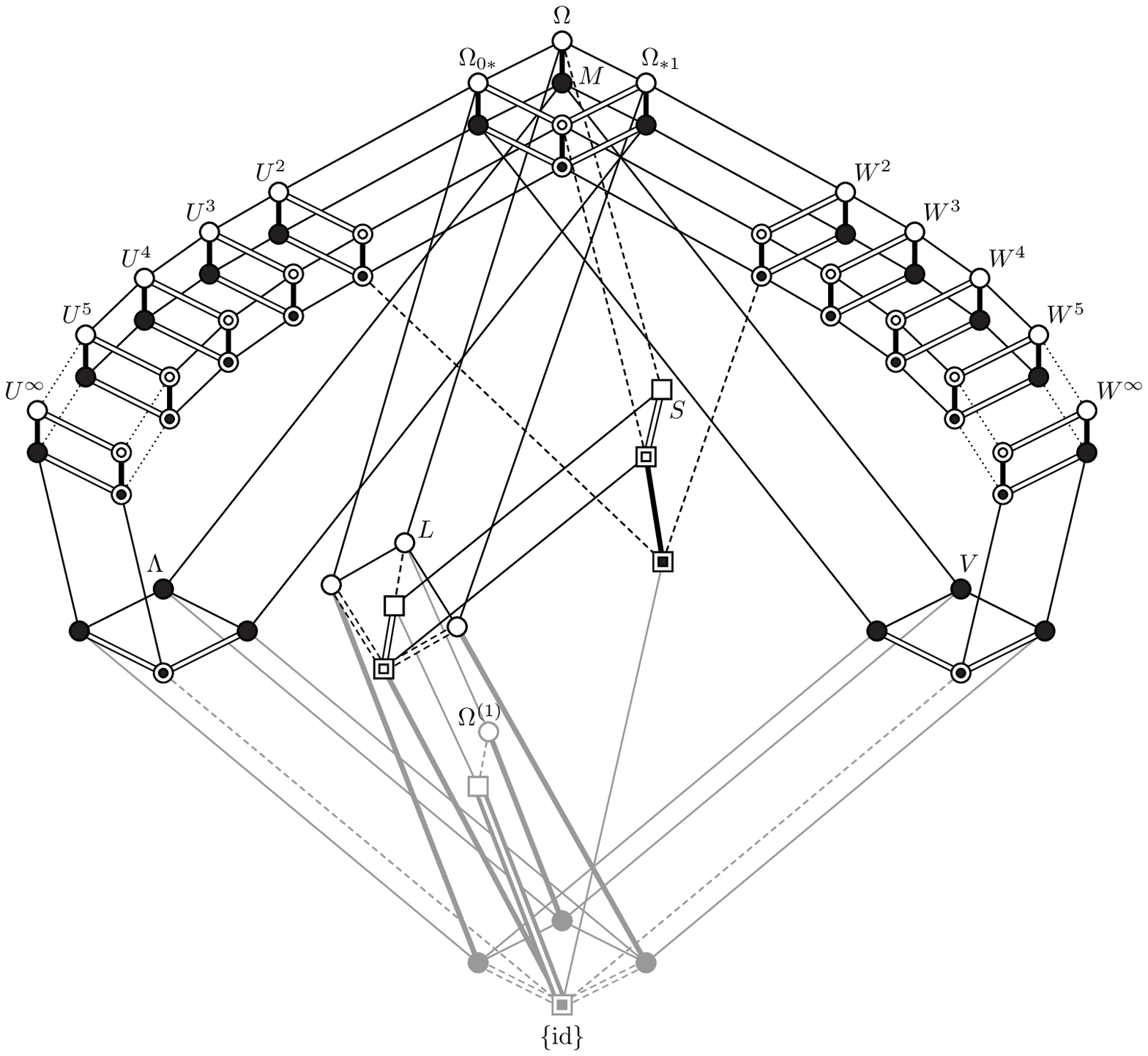}%
\caption{The Post lattice}%
\label{fig post}%
\end{center}
\end{figure}

\begin{itemize}
\item $\Omega_{0\ast}$ is the clone of $0$-preserving functions;

\item $\Omega_{\ast1}$ is the clone of $1$-preserving functions;

\item $M$ is the clone of monotone (order preserving) functions;

\item $S$ is the clone of self-dual functions, i.e., functions satisfying
\[
\lnot f\left(  \lnot x_{1},\ldots,\lnot x_{n}\right)  =f\left(  x_{1}%
,\ldots,x_{n}\right)  ;
\]

\item $L$ is the clone of linear functions, i.e., functions of the form
$x_{1}+\cdots+x_{n}+c$ with $n\geq0,c\in\left\{  0,1\right\}  $;

\item $\Lambda$ consists of conjunctions $x_{1}\wedge\cdots\wedge x_{n}%
$~$\left(  n\in\mathbb{N}\hspace{0cm}\right)  $ and the two constants $0,1$;

\item $V$ consists of disjunctions $x_{1}\vee\cdots\vee x_{n}~\left(
n\in\mathbb{N}\hspace{0cm}\right)  $ and the two constants $0,1$;

\item $\Omega^{\left(  1\right)  }$ is the clone of essentially at most unary functions;

\item $W^{k}$ is the clone of functions preserving the relation $\left\{
0,1\right\}  ^{k}\setminus\left\{  \mathbf{0}\right\}  $; it can be generated,
e.g., by the function $w_{k}$ of arity $k+2$ defined by\footnote{The lower
covers of $W^{k}$ in the Post lattice are $W^{k}\cap M,W^{k}\cap\Omega_{01}$
and $W^{k+1}$. Therefore, in order to verify that $w_{k}$ generates $W^{k}$,
it suffices to check that $w_{k}\in W^{k}$ and $w_{k}$ is neither monotone nor
idempotent, nor does it belong to $W^{k+1}$.}%
\[
w_{k}\left(  x_{1},\ldots,x_{k+2}\right)  =\left\{  \!\!%
\begin{array}
[c]{ll}%
0, & \text{if }\left\vert \left\{  i:x_{i}=1\right\}  \right\vert =2\text{ and
}x_{1}=1;\\
1, & \text{otherwise.}%
\end{array}
\right.
\]

\item $W^{\infty}=W^{2}\cap W^{3}\cap\cdots$ is the clone generated by implication;

\item $U^{k}$ is the dual of $W^{k}$ for $k=2,3,\ldots,\infty$.
\end{itemize}

All other clones can be obtained as intersections of these clones. The
different types of nodes and edges in Figure~\ref{fig post} help the
navigation in the Post lattice as follows:

\begin{itemize}
\item nodes representing clones of idempotent functions have a double outline
(others have a single outline), and a double edge connects a clone
$\mathcal{C}$ to $\mathcal{C}\cap\Omega_{01}$;

\item nodes representing clones of monotone functions are filled (others have
empty interior), and a thick edge connects a clone $\mathcal{C}$ to
$\mathcal{C}\cap M$;

\item nodes representing clones of self-dual functions are squares (others are
circles), and a dashed edge connects a clone $\mathcal{C}$ to $\mathcal{C}\cap
S$;

\item nodes representing clones of essentially at most unary functions are
grey (others are black), and all edges incident with unary clones are grey.
\end{itemize}

\subsection{Relational constraints\label{subsect relational constarints}}

If $P\subseteq A^{m}$ is a relation of arity $m$, and $N\in A^{m\times n}$ is
an $m\times n$ matrix such that each column of $N$ belongs to $P$, then we say
that $N$ is a $P$\emph{-matrix}. Applying an $n$-ary function $f$ to the rows
of $N$, we obtain the column vector $f\left(  N\right)  \in A^{m}$. A
\emph{relational constraint} of arity $m$ is a pair $\left(  P,Q\right)  $,
where $P$ and $Q$ are $m$-ary relations. An $n$-ary function $f$
\emph{satisfies} the constraint $\left(  P,Q\right)  $ if $f\left(  N\right)
\in Q$ for every $P$-matrix $N$ of size $m\times n$. Satisfaction of
relational constraints gives rise to a Galois connection that defines
equational classes of functions.

\begin{theorem}
[\cite{Pip}]\label{thm pippenger}A class of functions on a finite set $A$ is
an equational class iff it can be defined by relational constraints.
\end{theorem}

As an illustration of this theorem let us consider our two examples from
Subsection~\ref{subsect equational classes}: the class of antimonotone
functions can be defined by the constraint $\left(  \leq,\geq\right)  $, and
the class $\Omega_{=}$ can be defined by $\left(  \left\{  \left(  0,1\right)
\right\}  ,\left\{  \left(  0,0\right)  ,\left(  1,1\right)  \right\}
\right)  $.

Iterative algebras can be characterized by relational constraints as well. Let
us note that iterative algebras are usually defined with the help of the five
operations $\zeta,\tau,\Delta,\nabla,\ast$ introduced by Mal'cev \cite{M}, but
using function class composition we can give a very compact definition: a
function class $\mathcal{K}$ is an \emph{iterative algebra} iff $\mathcal{K}%
\circ\left(  \mathcal{K}\cup\left\{  \operatorname{id}\right\}  \right)
\subseteq\mathcal{K}$.

\begin{theorem}
[\cite{H}]A class of functions on a finite set $A$ is an iterative algebra iff
it can be defined by relational constraints $\left(  P,Q\right)  $ with
$Q\subseteq P$.
\end{theorem}

A function $f$ \emph{preserves} the relation $P$ iff $f$ satisfies the
constraint $\left(  P,P\right)  $. This induces the well-known Pol-Inv Galois
connection between clones and relational clones.

\begin{theorem}
[\cite{BKKR,G}]A class of functions on a finite set $A$ is a clone iff it can
be defined by relations.
\end{theorem}

Now we present another Galois connection that characterizes composition-closed
equational classes. Let us say that a function $f$ \emph{strongly satisfies}
the relational constraint $\left(  P,Q\right)  $, if $f$ satisfies both
$\left(  P,Q\right)  $ and $\left(  Q,Q\right)  $ (i.e., $f$ satisfies
$\left(  P,Q\right)  $ and preserves $Q$). The function class $\mathcal{K}$ is
\emph{strongly defined} by relational constraints if there exists a set
$\left\{  \left(  P_{i},Q_{i}\right)  :i\in I\right\}  $ of relational
constraints such that a function belongs to $\mathcal{K}$ iff it strongly
satisfies $\left(  P_{i},Q_{i}\right)  $ for all $i\in I$.

\begin{theorem}
\label{thm galois connection}A class of functions on a finite set $A$ is a
composition-closed equational class iff it can be strongly defined by
relational constraints.
\end{theorem}

\begin{proof}
First let us suppose that $\mathcal{K}$ is strongly defined by a set $\left\{
\left(  P_{i},Q_{i}\right)  :i\in I\right\}  $ of relational constraints,
where $P_{i}$ and $Q_{i}$ are relations of arity $m_{i}$. By
Theorem~\ref{thm pippenger}, $\mathcal{K}$ is an equational class. To verify
that $\mathcal{K}$ is closed under composition, let us consider arbitrary
functions $f,g_{1},\ldots,g_{n}\in\mathcal{K}$, where the arity of $f$ is $n$,
and the arity of $g_{1},\ldots,g_{n}$ is $k$. Then the composition $h=f\left(
g_{1},\ldots,g_{n}\right)  $ is a $k$-ary function on $A$. Let $i\in I$, and
let $N$ be a $P_{i}$-matrix of size $m_{i}\times k$. Since $g_{1},\ldots
,g_{n}$ satisfy the constraint $\left(  P_{i},Q_{i}\right)  $, the
$m_{i}\times n$ matrix $N^{\prime}$ formed by the column vectors $g_{1}\left(
N\right)  ,\ldots,g_{n}\left(  N\right)  $ is a $Q_{i}$-matrix. Therefore,
$h\left(  N\right)  =f\left(  N^{\prime}\right)  \in Q_{i}$, as $f$ preserves
the relation $Q_{i}$. This shows that the composition $h$ satisfies the
constraint $\left(  P_{i},Q_{i}\right)  $. Noting that $f,g_{1},\ldots,g_{n}$
all preserve the relation $Q_{i}$, we see that $h$ preserves $Q_{i}$ as well,
hence $h$ strongly satisfies $\left(  P_{i},Q_{i}\right)  $. This holds for
all $i\in I$, thus $h\in\mathcal{K}$, as claimed.

For the other implication, let us assume that $\mathcal{K}$ is a
composition-closed equational class. Let us write all elements of $A^{n}$
below each other (as row vectors), and let $O_{n}$ denote the resulting
$\left\vert A\right\vert ^{n}\times n$ matrix. Let $P_{n}$ be the set of
column vectors of $O_{n}$, and let $Q_{n}$ be the set of all column vectors of
the form $f\left(  O_{n}\right)  $, where $f\in\mathcal{K}$ is of arity $n$.
Let $\mathcal{K}^{\prime}$ be the class of functions strongly defined by
$\left\{  \left(  P_{n},Q_{n}\right)  :n\in\mathbb{N}\hspace{0cm}\right\}  $.
We will prove that $\mathcal{K}^{\prime}=\mathcal{K}$.

If $f^{\prime}\in\mathcal{K}^{\prime}$ is a function of arity $n$, then
$f^{\prime}\left(  O_{n}\right)  \in Q_{n}$, hence, according to the
definition of $Q_{n}$, there exists an $n$-ary function $f\in\mathcal{K}$ such
that $f^{\prime}\left(  O_{n}\right)  =f\left(  O_{n}\right)  $. Since the
rows $O_{n}$ contain every element of $A^{n}$, this implies that $f^{\prime
}=f$, thus $f^{\prime}\in\mathcal{K}$, hence we can conclude that
$\mathcal{K}^{\prime}\subseteq\mathcal{K}$.

In order to prove that $\mathcal{K}\subseteq\mathcal{K}^{\prime}$, we need to
verify for an arbitrary $n$-ary function $f\in\mathcal{K}$ that $f$ strongly
satisfies the constraint $\left(  P_{k},Q_{k}\right)  $ for every
$k\in\mathbb{N}\hspace{0cm}$. If $N_{1}$ is a $P_{k}$-matrix of size
$\left\vert A\right\vert ^{k}\times n$, then $f\left(  N_{1}\right)
=f_{1}\left(  O_{k}\right)  $ for a suitable $k$-ary subfunction $f_{1}$ of
$f$. Since $\mathcal{K}$ is an equational class, we have $f_{1}\in\mathcal{K}%
$, hence $f_{1}\left(  O_{k}\right)  \in Q_{k}$. This shows that $f$ satisfies
the constraint $\left(  P_{k},Q_{k}\right)  $. We need to check yet that $f$
preserves $Q_{k}$, i.e., that $f\left(  N_{2}\right)  \in Q_{k}$ for any
$Q_{k}$-matrix $N_{2}$ of size $\left\vert A\right\vert ^{k}\times n$. By the
definition of $Q_{k}$, the list of columns of $N_{2}$ is of the form
$g_{1}\left(  O_{k}\right)  ,\ldots,g_{n}\left(  O_{k}\right)  $ for some
$k$-ary $g_{1},\ldots,g_{n}\in\mathcal{K}$, thus $f\left(  N_{2}\right)
=h\left(  O_{k}\right)  $, where $h=f\left(  g_{1},\ldots,g_{n}\right)  $.
Since $\mathcal{K}$ is closed under composition, we have $h\in\mathcal{K}$,
hence $f\left(  N_{2}\right)  =h\left(  O_{k}\right)  \in Q_{k}$, as claimed.
\end{proof}

Concerning our two examples, let us observe that the class of antimonotone
functions is not closed under composition, but $\Omega_{=}$ is closed. Indeed,
$\Omega_{=}$ is strongly defined by the constraint $\left(  \left\{  \left(
0,1\right)  \right\}  ,\left\{  \left(  0,0\right)  ,\left(  1,1\right)
\right\}  \right)  $, since the relation $\left\{  \left(  0,0\right)
,\left(  1,1\right)  \right\}  $ is just the equality relation, and it is
preserved by every function.

\section{Idempotents vs. clones\label{sect idempotents vs clones}}

From now on we will restrict our attention to Boolean functions. In this
section we make some basic observations about the relationship between a
composition-closed equational class $\mathcal{K}$ and the clone $\mathcal{C}$
generated by $\mathcal{K}$. These observations will make it possible to
construct all such classes for a given clone $\mathcal{C}$.

Let us first briefly discuss the unary case: There are $16$ equational classes
consisting of essentially at most unary functions, and the following $10$ are
closed under composition:%
\[
\emptyset,~\left\{  0\right\}  ,~\left\{  1\right\}  ,~\left\{  0,1\right\}
,~\left\{  \operatorname{id}\right\}  ,~\left\{  0,\operatorname{id}\right\}
,~\left\{  1,\operatorname{id}\right\}  ,~\left\{  0,1,\operatorname{id}%
\right\}  ,~\left\{  \operatorname{id},\lnot\right\}  ,~\left\{
\operatorname{id},\lnot,0,1\right\}  .
\]
As we shall see in the following lemma, disregarding these 10 trivial cases,
we can always assume that $+,\lnot+,\rightarrow$ or $\lnot\rightarrow$ is
present in our composition-closed equational class.

\begin{lemma}
\label{lemma +imp}If $\mathcal{K}$ is a composition-closed equational class
that is not a clone, then

\begin{enumerate}
\item $\mathcal{K}\subseteq\Omega_{=}$;

\item if $\mathcal{K}\nsubseteq\Omega^{\left(  1\right)  }$, then
$\mathcal{K}\cap\left\{  +,\lnot+,\rightarrow,\lnot\rightarrow\right\}
\neq\emptyset$.
\end{enumerate}
\end{lemma}

\begin{proof}
If $\mathcal{K}$ is closed under composition, but $\mathcal{K}$ is not a
clone, then $\operatorname{id}\notin\mathcal{K}$. Since $\mathcal{K}$ is an
equational class, for each $f\in\mathcal{K}$ the unary subfunction $\Delta
_{f}\left(  x\right)  =f\left(  x,\ldots,x\right)  $ of $f$ belongs to
$\mathcal{K}$. If there is a function $f\in\mathcal{K}\cap\Omega_{01}$, then
we can conclude that $\operatorname{id}=\Delta_{f}\in\mathcal{K}$, a
contradiction. If there is a function $f\in\mathcal{K}\cap\Omega_{10}$, then
$\lnot=\Delta_{f}\in\mathcal{K}$, and since $\mathcal{K}$ is closed under
composition, we have $\operatorname{id}=\lnot\lnot\in\mathcal{K}$, which is a
contradiction again. Thus $\mathcal{K}\cap\Omega_{01}=\emptyset$ and
$\mathcal{K}\cap\Omega_{10}=\emptyset$, therefore $\mathcal{K}\subseteq
\Omega_{=}$.

To prove the second statement of the lemma, let us observe that any equational
class $\mathcal{K}\nsubseteq\Omega^{\left(  1\right)  }$ must contain at least
one of the 14 functions shown on the \textquotedblleft second
level\textquotedblright\ of Figure~\ref{fig subf}. Since we have
$\mathcal{K}\subseteq\Omega_{=}$, we see that at least one of the functions
$+,\lnot+,\rightarrow,\lnot\rightarrow$ belongs to $\mathcal{K}$.
\end{proof}

Note that the first statement of the lemma shows that the non-clone
composition-closed equational classes form a principal filter in $\mathbf{I}$
(see Figure~\ref{fig biglattice}). Hence the lattice $\mathbf{I}$ has six
coatoms: the five maximal clones ($\Omega_{0\ast},\Omega_{\ast1},M,S,L$) and
$\Omega_{=}$. Next we prove that every composition-closed equational class is
idempotent, as mentioned in Subsection~\ref{subsect equational classes}.

\begin{lemma}
\label{lemma elso}Let $\mathcal{K}_{1},\mathcal{K}_{2}$ be composition-closed
equational classes such that $\mathcal{K}_{1}\subseteq\mathcal{K}_{2}$ and
$\mathcal{K}_{1}\nsubseteq\Omega^{\left(  1\right)  }$. Then we have
$\mathcal{K}_{1}\circ\mathcal{K}_{2}=\mathcal{K}_{2}$.
\end{lemma}

\begin{proof}
One direction is obvious: $\mathcal{K}_{1}\circ\mathcal{K}_{2}\subseteq
\mathcal{K}_{2}\circ\mathcal{K}_{2}\subseteq\mathcal{K}_{2}$. The containment
$\mathcal{K}_{2}\subseteq\mathcal{K}_{1}\circ\mathcal{K}_{2}$ is also clear if
$\operatorname{id}\in\mathcal{K}_{1}$. If this is not the case, then one of
the functions $+,\lnot+,\rightarrow,\lnot\rightarrow$ belongs to
$\mathcal{K}_{1}$ by Lemma~\ref{lemma +imp}. If $f\left(  x,y\right)
=x+y\in\mathcal{K}_{1}$, then for any $g\in\mathcal{K}_{2}$ we have
$g=f\left(  f\left(  x,x\right)  ,g\right)  \in\mathcal{K}_{1}\circ
\mathcal{K}_{2}$, thus $\mathcal{K}_{2}\subseteq\mathcal{K}_{1}\circ
\mathcal{K}_{2}$. The same is true for $\lnot+$ and $\rightarrow$, while for
$f\left(  x,y\right)  =\lnot\left(  x\rightarrow y\right)  $ we can use the
composition $f\left(  g,f\left(  x,x\right)  \right)  $.
\end{proof}

\begin{proposition}
\label{prop (sub)idempotents}An equational class $\mathcal{K}$ is closed under
composition iff it is idempotent.
\end{proposition}

\begin{proof}
The \textquotedblleft if\textquotedblright\ part is obvious; the
\textquotedblleft only if\textquotedblright\ part is also obvious if
$\mathcal{K}\subseteq\Omega^{\left(  1\right)  }$, otherwise it follows from
the above lemma with $\mathcal{K}_{1}=\mathcal{K}_{2}=\mathcal{K}$.
\end{proof}

In light of the proposition above, in the following we will refer to a
composition-closed equational class simply as an idempotent. In the next
proposition we explore some relationships between an idempotent $\mathcal{K}$
and the clone $\left[  \mathcal{K}\right]  $ generated by $\mathcal{K}$ that
will play a crucial role in the sequel.

\begin{proposition}
\label{prop CK}Let $\mathcal{K}\nsubseteq\Omega^{\left(  1\right)  }$ be an
idempotent, and let $\mathcal{C}=\left[  \mathcal{K}\right]  $. Then we have
$\mathcal{C}\circ\mathcal{K}=\mathcal{K}$ and $\mathcal{K}\circ\mathcal{C}%
=\mathcal{C}$.
\end{proposition}

\begin{proof}
The second equality follows from Lemma \ref{lemma elso}. The containment
$\mathcal{C}\circ\mathcal{K}\supseteq\mathcal{K}$ is clear, since
$\operatorname{id}\in\mathcal{C}$, so it remains to prove that $\mathcal{C}%
\circ\mathcal{K}\subseteq\mathcal{K}$. Every element of $\mathcal{C}%
\circ\mathcal{K}$ is of the form $h=f\left(  g_{1},\ldots,g_{n}\right)  $,
where $f\in\mathcal{C}$ is $n$-ary, and $g_{1},\ldots,g_{n}\in\mathcal{K}$ are
$k$-ary functions. Since $f\in\mathcal{C}$, it is a composition of elements of
$\mathcal{K}$ and projections. We will prove $h\in\mathcal{K}$ by induction on
the size (number of nodes) of the tree describing this composition (see
Section~\ref{sect intro}).

If the tree has only one node, then $f\in\mathcal{K}$, and then $h\in
\mathcal{K}$ follows since $\mathcal{K}$ is closed under composition. If the
tree has at least two nodes, then $f$ is of the form $f=u\left(  v_{1}%
,\ldots,v_{r}\right)  $ with appropriate functions $u\in\mathcal{K}$ and
$v_{1},\ldots,v_{r}\in\mathcal{C}$, where the composition tree of each $v_{i}$
is smaller than the tree corresponding to $f$. Now we can write $h$ as%
\[
h=\left(  u\left(  v_{1},\ldots,v_{r}\right)  \right)  \left(  g_{1}%
,\ldots,g_{n}\right)  =u\left(  v_{1}\left(  g_{1},\ldots,g_{n}\right)
,\ldots,v_{r}\left(  g_{1},\ldots,g_{n}\right)  \right)  .
\]
By the induction hypothesis, every one of the functions $v_{i}\left(
g_{1},\ldots,g_{n}\right)  $ belongs to $\mathcal{K}$, therefore
$h\in\mathcal{K}\circ\mathcal{K}=\mathcal{K}$.
\end{proof}

For any equational class $\mathcal{K}$ let us write $\left\lfloor
\mathcal{K}\right\rfloor $ for the smallest idempotent containing
$\mathcal{K}$, and recall that $\left[  \mathcal{K}\right]  $ denotes the
smallest clone containing $\mathcal{K}$. The next proposition shows a
relationship between these two closure operators.

\begin{proposition}
\label{prop gen vs lgen}For any equational class $\mathcal{K}$ we have
$\left[  \mathcal{K}\right]  =\left\lfloor \mathcal{K}\cup\left\{
\operatorname{id}\right\}  \right\rfloor $ and $\left\lfloor \mathcal{K}%
\right\rfloor =\left[  \mathcal{K}\right]  \circ\mathcal{K}.$
\end{proposition}

\begin{proof}
The first equality expresses the obvious fact that the clone generated by
$\mathcal{K}$ consists of those functions that can be obtained from elements
of $\mathcal{K}$ and from projections by means of composition. Introducing the
notation $\mathcal{C}=\left[  \mathcal{K}\right]  $, the second equality takes
the form $\left\lfloor \mathcal{K}\right\rfloor =\mathcal{C}\circ\mathcal{K}$.
The right hand side contains $\mathcal{K}$, since $\operatorname{id}%
\in\mathcal{C}$, and it is closed under composition (hence idempotent):%
\[
\left(  \mathcal{C}\circ\mathcal{K}\right)  \circ\left(  \mathcal{C}%
\circ\mathcal{K}\right)  \subseteq\mathcal{C}\circ\mathcal{C}\circ
\mathcal{C}\circ\mathcal{K}=\mathcal{C}\circ\mathcal{K}.
\]

To prove that $\mathcal{C}\circ\mathcal{K}$ is indeed the smallest idempotent
containing $\mathcal{K}$, we have to show that $\mathcal{C}\circ
\mathcal{K}\subseteq\left\lfloor \mathcal{K}\right\rfloor $. The case
$\mathcal{K}\subseteq\Omega^{\left(  1\right)  }$ is trivial, otherwise we can
apply Proposition~\ref{prop CK} to the idempotent $\left\lfloor \mathcal{K}%
\right\rfloor $, and we get%
\[
\mathcal{C}\circ\left\lfloor \mathcal{K}\right\rfloor =\left[  \left\lfloor
\mathcal{K}\right\rfloor \right]  \circ\left\lfloor \mathcal{K}\right\rfloor
=\left\lfloor \mathcal{K}\right\rfloor .
\]
Therefore, $\mathcal{C}\circ\mathcal{K}\subseteq\mathcal{C}\circ\left\lfloor
\mathcal{K}\right\rfloor =\left\lfloor \mathcal{K}\right\rfloor $, and this
proves the proposition.
\end{proof}

\begin{example}
Let us give some simple examples of idempotents together with relational
constraints strongly defining them.

\begin{itemize}
\item $\Omega_{=}=\left\{  f\in\Omega:f\left(  \mathbf{0}\right)  =f\left(
\mathbf{1}\right)  \right\}  $

strongly defined by $\left(  \left\{  \left(  0,1\right)  \right\}  ,\left\{
\left(  0,0\right)  ,\left(  1,1\right)  \right\}  \right)  $

\item $\Omega_{00}=\left\{  f\in\Omega:f\left(  \mathbf{0}\right)  =f\left(
\mathbf{1}\right)  =0\right\}  $

strongly defined by $\left(  \left\{  \left(  0,1\right)  \right\}  ,\left\{
\left(  0,0\right)  \right\}  \right)  $

\item $\Omega_{11}=\left\{  f\in\Omega:f\left(  \mathbf{0}\right)  =f\left(
\mathbf{1}\right)  =1\right\}  $

strongly defined by $\left(  \left\{  \left(  0,1\right)  \right\}  ,\left\{
\left(  1,1\right)  \right\}  \right)  $

\item $\mathcal{R}=\left\{  f\in\Omega:f\left(  \lnot x_{1},\ldots,\lnot
x_{n}\right)  =f\left(  x_{1},\ldots,x_{n}\right)  \right\}  $

strongly defined by $\left(  \left\{  \left(  0,1\right)  ,\left(  1,0\right)
\right\}  ,\left\{  \left(  0,0\right)  ,\left(  1,1\right)  \right\}
\right)  $
\end{itemize}

\noindent Members of the class $\mathcal{R}$ are called \emph{reflexive
functions}.
\end{example}

\begin{example}
\label{ex lgen+}One can determine $\left\lfloor +\right\rfloor $ using our
observations about the composition-closed equational class generated by the
addition operation of a field (see Section~\ref{sect intro}). Alternatively,
we can use Proposition~\ref{prop gen vs lgen}: the elements of $\left\lfloor
+\right\rfloor $ are exactly the functions of the form%
\[
f\left(  x_{i_{1}}+x_{j_{1}},\ldots,x_{i_{n}}+x_{j_{n}}\right)  ,
\]
where $f\in\left[  +\right]  =L_{0\ast}$ and $x_{i_{k}},x_{j_{k}}$ are
arbitrary (not necessarily distinct) variables. Such a function is a sum of an
even number of variables, hence
\[
\left\lfloor +\right\rfloor =\left\{  x_{1}+\cdots+x_{n}:n\text{ is
even}\right\}  =L_{00}.
\]
Similarly, we have
\[
\left\lfloor \lnot+\right\rfloor =\left\{  x_{1}+\cdots+x_{n}+1:n\text{ is
even}\right\}  =L_{11}.
\]

\end{example}

\begin{example}
Proposition \ref{prop gen vs lgen} gives the following general form for the
elements of $\left\lfloor \rightarrow\right\rfloor $:%
\[
f\left(  x_{i_{1}}\rightarrow x_{j_{1}},\ldots,x_{i_{n}}\rightarrow x_{j_{n}%
}\right)  ,
\]
where $f\in\left[  \rightarrow\right]  =W^{\infty}$, and $x_{i_{k}},x_{j_{k}}$
are arbitrary (not necessarily distinct) variables. Here, the lack of
associativity makes it harder to decide whether a given function is of this
form. For example, the function $f\left(  x,y,z\right)  =x\rightarrow\left(
y\rightarrow z\right)  $ seems not to have the required form, but we may
rewrite it as $\left(  x\rightarrow y\right)  \rightarrow\left(  x\rightarrow
z\right)  $, therefore $f\in\left\lfloor \rightarrow\right\rfloor $. On the
other hand, $g\left(  x,y,z\right)  =\left(  x\rightarrow y\right)
\rightarrow z$ does not belong to $\left\lfloor \rightarrow\right\rfloor $. To
verify this, let us observe that $\rightarrow\in\Omega_{11}$, hence
$\left\lfloor \rightarrow\right\rfloor \subseteq\Omega_{11}$. However,
$g\notin\Omega_{11}$ as $g\left(  0,0,0\right)  =0$, therefore $g\notin
\left\lfloor \rightarrow\right\rfloor $. In
Proposition~\ref{prop bottom W_infty} we will describe $\left\lfloor
\rightarrow\right\rfloor $ by relational constraints; this description makes
it easy to decide membership for $\left\lfloor \rightarrow\right\rfloor $.
\end{example}

Proposition \ref{prop gen vs lgen} suggests the following strategy for finding
all idempotents: Let us fix a clone $\mathcal{C}$, and form all possible
compositions $\mathcal{C}\circ\mathcal{K}$, where $\mathcal{K}$ is an
equational class, such that $\left[  \mathcal{K}\right]  =\mathcal{C}${}. This
way we obtain all those idempotents that generate the clone $\mathcal{C}$. Let
us denote the set of these idempotents by $\mathbf{I}\left(  \mathcal{C}%
\right)  $:%
\[
\mathbf{I}\left(  \mathcal{C}\right)  =\left\{  \mathcal{K}\in\mathbf{I}%
:\left[  \mathcal{K}\right]  =\mathcal{C}\right\}  .
\]
Clearly these sets form a partition of $\mathbf{I}$. For unary clones
$\mathcal{C}$ it is an easy task to determine $\mathbf{I}\left(
\mathcal{C}\right)  $: for $\mathcal{C}\subseteq\left\{  0,1,\operatorname{id}%
\right\}  $ we have $\mathbf{I}\left(  \mathcal{C}\right)  =\left\{
\mathcal{C},\mathcal{C}_{=}\right\}  $; for the other unary clones we have
$\mathbf{I}\left(  \mathcal{C}\right)  =\left\{  \mathcal{C}\right\}  $.

\begin{theorem}
\label{thm interval}For any clone $\mathcal{C}$, the set $\mathbf{I}\left(
\mathcal{C}\right)  $ is an interval in the lattice of idempotents. If
$\mathbf{I}\left(  \mathcal{C}\right)  $ has more than one element, then
$\mathcal{C}_{=}$ is its only coatom.
\end{theorem}

\begin{proof}
We have settled the unary case above, so now let us suppose that
$\mathcal{C}\nsubseteq\Omega^{\left(  1\right)  }$. It is clear that the
largest element of $\mathbf{I}\left(  \mathcal{C}\right)  $ is $\mathcal{C}$
itself, and it is also clear that $\mathbf{I}\left(  \mathcal{C}\right)  $ is
convex. Therefore, in order to prove that it is indeed an interval, it
suffices to show that $\mathbf{I}\left(  \mathcal{C}\right)  $ is closed under
arbitrary intersections, hence it has a least element.

Let $\mathcal{K}_{i}\in\mathbf{I}\left(  \mathcal{C}\right)  ~\left(  i\in
I\right)  $, and let $\mathcal{K}=\bigcap\mathcal{K}_{i}$. Then $\mathcal{C}%
\circ\mathcal{K}\subseteq\bigcap\left(  \mathcal{C}\circ\mathcal{K}%
_{i}\right)  =\bigcap\mathcal{K}_{i}=\mathcal{K}$ by Proposition~\ref{prop CK}%
. On the other hand, $\operatorname{id}\in\mathcal{C}$ implies that
$\mathcal{C}\circ\mathcal{K}\supseteq\mathcal{K}$, so we have $\mathcal{C}%
\circ\mathcal{K}=\mathcal{K}$. Composing by $\left[  \mathcal{K}\right]  $ on
the right, we get $\mathcal{C}\circ\mathcal{K}\circ\left[  \mathcal{K}\right]
=\mathcal{K}\circ\left[  \mathcal{K}\right]  $, and using
Proposition~\ref{prop CK} once more, we can conclude that $\mathcal{C}%
\circ\left[  \mathcal{K}\right]  =\left[  \mathcal{K}\right]  $. The left hand
side contains $\mathcal{C}$, since $\operatorname{id}\in\left[  \mathcal{K}%
\right]  $, so we have $\mathcal{C} \subseteq\left[  \mathcal{K}\right]  $.
The containment $\left[  \mathcal{K}\right]  \subseteq\mathcal{C}$ is obvious,
so we see that $\left[  \mathcal{K}\right]  =\mathcal{C}$, and therefore
$\mathcal{K}\in\mathbf{I}\left(  \mathcal{C}\right)  $, as claimed. (Note that
what we have proved is rather surprising:\ if each $\mathcal{K}_{i}$ contains
a generating set of the clone $\mathcal{C}$, then so does $\bigcap
\mathcal{K}_{i}$.)

The statement about the unique coatom follows from Lemma~\ref{lemma +imp}: if
$\mathcal{K}$ is an element of $\mathbf{I}\left(  \mathcal{C}\right)  $ that
is different from $\mathcal{C}$, then $\mathcal{K}$ is a composition-closed
equational class that is not a clone, hence $\mathcal{K}\subseteq
\mathcal{C}\cap\Omega_{=}=\mathcal{C}_{=}$.
\end{proof}

The contents of the above theorem are represented in Figure~\ref{fig typint},
which shows a picture of a typical nontrivial interval $\mathbf{I}\left(
\mathcal{C}\right)  $. As the next corollary shows, Theorem~\ref{thm interval}
allows us to prove that $\mathbf{I}\left(  \mathcal{C}\right)  $ is trivial
for \textquotedblleft many\textquotedblright\ Boolean clones.%
\begin{figure}
[ptb]
\begin{center}
\includegraphics[
height=1.6587in,
width=0.5065in
]%
{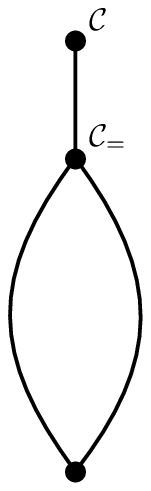}%
\caption{A typical nontrivial interval $\mathbf{I}\left(  \mathcal{C}\right)
$}%
\label{fig typint}%
\end{center}
\end{figure}

\begin{corollary}
If $\mathcal{C}$ is a clone different from $\left\{  \operatorname{id}%
\right\}  ,\allowbreak\left\{  0,\operatorname{id}\right\}  ,\allowbreak
\left\{  1,\operatorname{id}\right\}  ,\allowbreak\left\{
0,1,\operatorname{id}\right\}  ,\allowbreak\Omega,\allowbreak\Omega_{0\ast
},\allowbreak\Omega_{\ast1},\allowbreak L,\allowbreak L_{0\ast},\allowbreak
L_{\ast1},\allowbreak W^{2},\ldots,W^{\infty},\allowbreak U^{2},\ldots
,U^{\infty}$, then $\mathbf{I}\left(  \mathcal{C}\right)  =\left\{
\mathcal{C}\right\}  $.
\end{corollary}

\begin{proof}
If $\mathbf{I}\left(  \mathcal{C}\right)  $ has at least two elements, then
$\mathcal{C}_{=}\in\mathbf{I}\left(  \mathcal{C}\right)  $ according to
Theorem~\ref{thm interval}. This implies that $\mathcal{C}_{=}$ generates the
clone $\mathcal{C}$. However, for many clones, $\mathcal{C}_{=}=\mathcal{C}%
\cap\Omega_{=}$ is too small to generate $\mathcal{C}$. Indeed, $S\cap
\Omega_{=}=\Omega_{01}\cap\Omega_{=}=\emptyset$ and $M\cap\Omega_{=}=\left\{
0,1\right\}  ,$ therefore (ignoring the unary clones) $\mathbf{I}\left(
\mathcal{C}\right)  $ is trivial unless $\mathcal{C}$ is one of the clones
represented in Figure~\ref{fig post} by a black circle with a single outline
and an empty (non-filled) interior.
\end{proof}

Up to duality, only the cases $\mathcal{C}=\allowbreak\Omega,\allowbreak
\Omega_{\ast1},\allowbreak L,\allowbreak L_{0\ast},\allowbreak W^{2}%
,\ldots,W^{\infty}$ remain. The first four cases are treated in
Section~\ref{sect easy}; we will see that in these cases $\mathbf{I}\left(
\mathcal{C}\right)  $ has at most $3$ elements. In Section~\ref{sect W_infty}
we give a characterization of the elements of $\mathbf{I}\left(  W^{\infty
}\right)  $, and we will prove that this interval is uncountable. We consider
$\mathbf{I}\left(  W^{k}\right)  $ for finite $k$ in Section~\ref{sect W_k},
and we will show that these intervals are finite, but their cardinalities do
not have a common upper bound.

\section{Idempotents corresponding to $\Omega,\Omega_{\ast1},L,L_{0\ast}%
$\label{sect easy}}

As explained in the previous section, $\mathbf{I}\left(  \mathcal{C}\right)  $
can be determined by computing all possible compositions of the form
$\mathcal{C}\circ\mathcal{K}$, where $\mathcal{K}$ is an equational class
generating the clone $\mathcal{C}$. In the following two lemmas we compute six
such compositions for $\mathcal{C}{}=\Omega$ and $\mathcal{C}{}=\Omega_{\ast
1}$, and then we will show that these suffice to determine the intervals
$\mathbf{I}\left(  \Omega\right)  $ and $\mathbf{I}\left(  \Omega_{\ast
1}\right)  $.

\begin{lemma}
\label{lemma Om.->}The following equalities hold:

\begin{enumerate}
\item $\Omega\circ\left\{  \rightarrow\right\}  =\Omega\circ\left\{
\lnot\rightarrow\right\}  =\Omega_{=}$;

\item $\Omega_{\ast1}\circ\left\{  \rightarrow\right\}  =\Omega_{11}$.
\end{enumerate}
\end{lemma}

\begin{proof}
(1) We prove only that $\Omega\circ\left\{  \rightarrow\right\}  =\Omega_{=}$;
the other case is similar. The functions in $\Omega\circ\left\{
\rightarrow\right\}  $ are of the following form:%
\begin{equation}
h\left(  x_{1},\ldots,x_{k}\right)  =f\left(  x_{i_{1}}\rightarrow x_{j_{1}%
},\ldots,x_{i_{n}}\rightarrow x_{j_{n}}\right)  , \label{Om.->}%
\end{equation}
where $f\in\Omega$ and $i_{1},j_{1},\ldots,i_{n},j_{n}\in\left\{
1,\ldots,k\right\}  $. Since $0\rightarrow0=1\rightarrow1$, every function of
this form belongs to $\Omega_{=}$. Conversely, for any given $h\in\Omega_{=}$
we will construct a suitable $f$ so that (\ref{Om.->}) holds. Let us choose
$n=k^{2}$, and let $\left\{  \left(  i_{1},j_{1}\right)  ,\ldots,\left(
i_{n},j_{n}\right)  \right\}  =\left\{  1,\ldots,k\right\}  \times\left\{
1,\ldots,k\right\}  $. Then we can rewrite (\ref{Om.->}) as $h\left(
\mathbf{a}\right)  =f\left(  \widehat{\mathbf{a}}\right)  $ for all
$\mathbf{a}=\left(  a_{1},\ldots,a_{k}\right)  \in\left\{  0,1\right\}  ^{k}$,
where $\widehat{\mathbf{a}}$ denotes the vector formed from all possible
implications between entries of $\mathbf{a}$. Now we can treat this as the
definition of $f$:%
\[
f\left(  \mathbf{b}\right)  :=\left\{  \!\!%
\begin{array}
[c]{ll}%
h\left(  \mathbf{a}\right)  , & \text{if }\mathbf{b}=\widehat{\mathbf{a}%
}\text{;}\\
0, & \text{if }\nexists\mathbf{a}\in\left\{  0,1\right\}  ^{k}:\mathbf{b}%
=\widehat{\mathbf{a}}.
\end{array}
\right.
\]
All we need to show is that $f$ is well-defined, i.e., that $\mathbf{b}%
=\widehat{\mathbf{a}}$ determines $h\left(  \mathbf{a}\right)  $ uniquely.

If $\widehat{\mathbf{a}}=\left(  1,\ldots,1\right)  $, then $a_{1}%
=\cdots=a_{k}$, but we cannot tell whether $\mathbf{a}=\mathbf{0}$ or
$\mathbf{a}=\mathbf{1}$. However, since $h\in\Omega_{=}$, the value of
$h\left(  \mathbf{a}\right)  $ is the same in both cases. Now let us suppose
that at least one entry of $\widehat{\mathbf{a}}$ is $0$, say, $a_{1}%
\rightarrow a_{2}=0$. Then we can infer immediately that $a_{1}=1$ and
$a_{2}=0$. Using this information we can recover the vector $\mathbf{a}$,
since $a_{i}\rightarrow a_{2}=\lnot a_{i}$ for all $i\in\left\{
1,\ldots,k\right\}  $. Thus in this case $\widehat{\mathbf{a}}$ uniquely
determines $\mathbf{a}$, hence $h\left(  \mathbf{a}\right)  $ is uniquely
determined as well.

(2) We can proceed similarly as above, with $h\in\Omega_{11}$ and $f\in
\Omega_{\ast1}$.
\end{proof}

\begin{lemma}
\label{lemma Om.+}The following equalities hold:

\begin{enumerate}
\item $\Omega\circ\left\{  +\right\}  =\Omega\circ\left\{  \lnot+\right\}
=\mathcal{R}$;

\item $\Omega_{\ast1}\circ\left\{  \lnot+\right\}  =\mathcal{R}_{11}$.
\end{enumerate}
\end{lemma}

\begin{proof}
We can use the same argument as in the previous proposition. For example,
considering $\Omega\circ\left\{  +\right\}  =\mathcal{R}$, we have%
\[
h\left(  x_{1},\ldots,x_{k}\right)  =f\left(  x_{i_{1}}+x_{j_{1}}%
,\ldots,x_{i_{n}}+x_{j_{n}}\right)  ,
\]
in place of (\ref{Om.->}), and we define $\widehat{\mathbf{a}}\in\left\{
0,1\right\}  ^{k^{2}}$ to be the vector formed from all sums of two entries of
$\mathbf{a}$. Now $\widehat{\mathbf{a}}$ tells us which entries of
$\mathbf{a}$ are the same, but does not tell us which one is $0$ and which one
is $1$. Therefore $\widehat{\mathbf{a}}$ determines $\mathbf{a}$ up to
negation, and this is equivalent to the condition $h\in\mathcal{R}$.
\end{proof}

\begin{theorem}
If $\mathcal{C}{}=\Omega$ or $\mathcal{C}{}=\Omega_{\ast1}$, then
$\mathbf{I}\left(  \mathcal{C}\right)  $ is a three-element chain:
$\mathbf{I}\left(  \mathcal{C}\right)  =\left\{  \mathcal{C},\mathcal{C}%
_{=},\mathcal{C}\cap\mathcal{R}\right\}  $.
\end{theorem}

\begin{proof}
Let us consider first $\mathcal{C}{}=\Omega$. If $\mathcal{K}{}\in
\mathbf{I}\left(  \Omega\right)  $, then $\Omega\circ\mathcal{K}=\mathcal{K}$
by Proposition~\ref{prop CK}. If $\mathcal{K}$ is a clone, then clearly we
have $\mathcal{K}{}=\Omega$. If this is not the case, then $\mathcal{K}%
{}\subseteq\Omega_{=}$, and at least one of the operations $+,\lnot
\rightarrow,\lnot+,\rightarrow$ belongs to $\mathcal{K}$ by
Lemma~\ref{lemma +imp}. If $\rightarrow\in\mathcal{K}{}$ or $\lnot
\rightarrow\in\mathcal{K}{}$, then $\mathcal{K}{}=\Omega\circ\mathcal{K}%
{}\supseteq\Omega_{=}$ by Lemma~\ref{lemma Om.->}, hence $\mathcal{K}{}%
=\Omega_{=}$. In the remaining cases the binary operations in $\mathcal{K}$
are $+$ and/or $\lnot+$, hence $\mathcal{K}=\Omega\circ\mathcal{K}%
\supseteq\mathcal{R}{}$ by Lemma~\ref{lemma Om.+}. On the other hand, in these
cases every binary operation in $\mathcal{K}{}$ is reflexive. We will show
that this implies $\mathcal{K}{}\subseteq\mathcal{R}{}$.

If $f\left(  x_{1},\ldots,x_{n}\right)  \in\mathcal{K}{}$ and $\mathbf{a}%
\in\left\{  0,1\right\}  ^{n}$, then $f\left(  \mathbf{a}\right)  =g\left(
0,1\right)  $, where $g\left(  x,y\right)  $ is obtained from $f\left(
x_{1},\ldots,x_{n}\right)  $ by replacing $x_{i}$ by $x$ or $y$ depending on
whether $a_{i}=0$ or $a_{i}=1$. Since $g$ is a subfunction of $f$, we have
$g\in\mathcal{K}{}$, and then the reflexivity of $g$ implies $f\left(
\lnot\mathbf{a}\right)  =g\left(  1,0\right)  =g\left(  0,1\right)  =f\left(
\mathbf{a}\right)  $. This is true for all $\mathbf{a\in}\left\{  0,1\right\}
^{n}$, thus $f$ is reflexive, as claimed.

The above considerations show that the only possible elements of
$\mathbf{I}\left(  \Omega\right)  $ are $\Omega,\Omega_{=}$ and $\mathcal{R}$.
It is easily verified that each of these three classes indeed generate the
clone $\Omega$, hence $\mathbf{I}\left(  \Omega\right)  =\left\{
\Omega,\Omega_{=},\mathcal{R}\right\}  $.

If $\mathcal{C}=\Omega_{\ast1}$, then $+,\lnot\rightarrow$ cannot belong to
$\mathcal{K}{}$. A similar argument as the above one shows that $\mathcal{K}%
{}=\Omega_{\ast1}$ if $\operatorname{id}\in\mathcal{K}{}$, $\mathcal{K}%
{}=\Omega_{\ast1}\cap\Omega_{=}=\Omega_{11}$ if $\operatorname{id}%
\notin\mathcal{K}{}$ but $\rightarrow\in\mathcal{K}{}$, and $\mathcal{K}%
{}=\Omega_{\ast1}\cap\mathcal{R}=\mathcal{R}_{11}$ otherwise. Thus we have
$\mathbf{I}\left(  \Omega_{\ast1}\right)  =\left\{  \Omega_{\ast1},\Omega
_{11},\mathcal{R}_{11}\right\}  $.
\end{proof}

As one can expect, the case of linear functions is considerably simpler.

\begin{theorem}
If $\mathcal{C}=L$ or $\mathcal{C}=L_{0\ast}$ then $\mathbf{I}\left(
\mathcal{C}\right)  $ is a two-element chain: $\mathbf{I}\left(
\mathcal{C}\right)  =\left\{  \mathcal{C},\mathcal{C}_{=}\right\}  $.
\end{theorem}

\begin{proof}
If $\mathcal{K}{}\in\mathbf{I}\left(  L_{0\ast}\right)  \setminus\left\{
L_{0\ast}\right\}  $, then $\mathcal{K}{}\subseteq L_{0\ast}\cap\Omega_{=}$ by
Theorem~\ref{thm interval}. However, $L_{0\ast}\cap\Omega_{=}=L_{00}%
=\left\lfloor +\right\rfloor $ (see Example~\ref{ex lgen+}), so we can
conclude that $\mathcal{K}{}\subseteq\left\lfloor +\right\rfloor $. On the
other hand, since $\mathcal{K}{}\subseteq\Omega_{00}$, it contains at least
one of the functions $+,\lnot\rightarrow$ by Lemma~\ref{lemma +imp}. Clearly
$\lnot\rightarrow\notin\mathcal{K}$, so we must have $+\in\mathcal{K}{}$, and
then $\left\lfloor +\right\rfloor \subseteq\mathcal{K}{}$. Thus $\mathcal{K}%
{}=\left\lfloor +\right\rfloor =L_{00}$, hence $\mathbf{I}\left(  L_{0\ast
}\right)  =\left\{  L_{0\ast},L_{00}\right\}  $.

Now let $\mathcal{K}{}\in\mathbf{I}\left(  L\right)  \setminus\left\{
L\right\}  $. Then $\mathcal{K}{}\subseteq L\cap\Omega_{=}=L_{00}\cup
L_{11}=\left\lfloor +\right\rfloor \cup\left\lfloor \lnot+\right\rfloor $.
Similarly to the previous case we see that $\left\lfloor +\right\rfloor
\subseteq\mathcal{K}$ or $\left\lfloor \lnot+\right\rfloor \subseteq
\mathcal{K}$. If only one of these held, then $\mathcal{K}{}$ would be a
subset of either $\Omega_{00}$ or $\Omega_{11}$, contradicting $\left[
\mathcal{K}\right]  =L$. So we must have $\mathcal{K}=\left\lfloor
+\right\rfloor \cup\left\lfloor \lnot+\right\rfloor =L_{=}$, hence
$\mathbf{I}\left(  L\right)  =\left\{  L,L_{=}\right\}  $.
\end{proof}

\section{Idempotents corresponding to $W^{\infty}$\label{sect W_infty}}

In the above cases for each given clone $\mathcal{C}$, the interval
$\mathbf{I}\left(  \mathcal{C}\right)  $ was finite. For $\mathcal{C}%
=W^{\infty}$ the situation is more complicated: we will prove in this section
that there are continuously many idempotents generating the clone $W^{\infty}%
$. In the next section we will see that the intervals $\mathbf{I}\left(
W^{k}\right)  $ are finite, but their sizes tend to infinity as $k\rightarrow
\infty$. In these two sections we will often work with the set of zeros (set
of \textquotedblleft false points\textquotedblright) of functions, so let us
set up some notation for this.

For an $n$-ary Boolean function $f$, let $f^{-1}\left(  0\right)  =\left\{
\mathbf{a}\in\left\{  0,1\right\}  ^{n}:f\left(  \mathbf{a}\right)
=0\right\}  $. We will often treat this set as a matrix: writing the elements
of $f^{-1}\left(  0\right)  $ as row vectors below each other, we obtain an
$m\times n$ matrix, where $m=\left\vert f^{-1}\left(  0\right)  \right\vert $.
The order of the rows is irrelevant, moreover, since we do not need to
distinguish between equivalent functions, we may rearrange the columns as well
(by permuting variables). Let us observe that $f\in W^{\infty}$ iff
$f^{-1}\left(  0\right)  $ has a constant $0$ column, and $f\in W^{k}$ iff
every matrix formed from at most $k$ rows of $f^{-1}\left(  0\right)  $ has a
constant $0$ column.

Next we introduce some operators on function classes that deal with zeros of
functions. For $k\geq2$ and $\mathcal{K}\subseteq\Omega$ let $Z_{k}%
\mathcal{K}$ denote the set of functions $f\in\Omega$ such that for every at
most $k$-element subset $H\subseteq f^{-1}\left(  0\right)  $ there exists a
function $g\in\mathcal{K}$ of the same arity as $f$ with $H\subseteq
g^{-1}\left(  0\right)  $. Furthermore, let $Z_{\infty}\mathcal{K}$ be the set
of functions $f\in\Omega$ such that there exists a function $g\in\mathcal{K}$
of the same arity as $f$ with $f^{-1}\left(  0\right)  \subseteq g^{-1}\left(
0\right)  $. Clearly, $Z_{2},Z_{3},\ldots,Z_{\infty}$ are closure operators on
$\Omega$, and for any $\mathcal{K}\subseteq\Omega$ we have%
\begin{equation}
Z_{2}\mathcal{K}\supseteq Z_{3}\mathcal{K}\supseteq\cdots\supseteq Z_{\infty
}\mathcal{K}=\bigcap_{k\geq2}\left(  Z_{k}\mathcal{K}\right)  .
\label{eq Z_2...Z_infty}%
\end{equation}

Proposition~\ref{prop W_infty} below describes $\mathbf{I}\left(  W^{\infty
}\right)  $ with the help of the operator $Z_{\infty}$.

\begin{lemma}
\label{lemma ize}If $\mathcal{K}$ is an idempotent equational class such that
$\left[  \mathcal{K}\right]  =W^{k}$ for some $k\in\left\{  2,3,\ldots
,\infty\right\}  $, then $Z_{\infty}\mathcal{K}=\mathcal{K}$.
\end{lemma}

\begin{proof}
For $\mathcal{K}=W^{k}$ the statement is clear, so let us suppose that
$\mathcal{K}\in\mathbf{I}\left(  W^{k}\right)  \setminus\left\{
W^{k}\right\}  $. Let $f\in\Omega,g\in\mathcal{K}$ be $n$-ary functions such
that $f^{-1}\left(  0\right)  \subseteq g^{-1}\left(  0\right)  $. We have to
show that $f\in\mathcal{K}$. We can suppose without loss of generality that
$\left\vert g^{-1}\left(  0\right)  \setminus f^{-1}\left(  0\right)
\right\vert =1$, i.e., $f$ and $g$ differ only at one position $\mathbf{a}%
\in\left\{  0,1\right\}  ^{n}$. (Otherwise we apply this several times
changing the values of $g$ one by one until we reach the desired function
$f$.) For notational simplicity let us also assume that $a_{1}=\cdots=a_{l}=0$
and $a_{l+1}=\cdots=a_{n}=1$. Then $g\left(  0,\ldots,0,1,\ldots,1\right)  =0$
and $f\left(  0,\ldots,0,1,\ldots,1\right)  =1$. (Here, and in the rest of the
proof all $n$-tuples are split into two parts of size $l$ and $n-l$.)

From Proposition~\ref{prop CK} we know that $\mathcal{K}=\left[
\mathcal{K}\right]  \circ\mathcal{K}=W^{k}\circ\mathcal{K}\supseteq W^{\infty
}\circ\mathcal{K}$, so it suffices to show that $f\in W^{\infty}%
\circ\mathcal{K}$. We consider the following function of arity $N+1$ with
$N=l\left(  n-l\right)  $, which clearly belongs to $W^{\infty}$:%
\[
h\left(  z_{1},\ldots,z_{N},w\right)  =\Bigl(\,%
{\displaystyle\bigvee\limits_{1\leq k\leq N}}
z_{k}\Bigr)\rightarrow w\text{.}%
\]
Let us construct the following function $f^{\prime}\in W^{\infty}%
\circ\mathcal{K}$, which is a composition of $h$ (as outer function) and
several subfunctions of $g$ (as inner functions):%
\begin{multline*}
f^{\prime}\left(  x_{1},\ldots,x_{l},y_{1},\ldots,y_{n-l}\right) \\
=\Bigl(%
{\displaystyle\bigvee\limits_{\substack{1\leq i\leq l\\1\leq j\leq n-l}}}
g\left(  x_{i},\ldots,x_{i},y_{j},\ldots,y_{j}\right)  \Bigr)\rightarrow
g\left(  x_{1},\ldots,x_{l},y_{1},\ldots,y_{n-l}\right)  \text{.}%
\end{multline*}

We claim that $f^{\prime}=f$. In order to verify this fact, let us observe
that since $g\left(  0,\ldots,0,1,\ldots,1\right)  =0$, we have $g\left(
1,\ldots,1,0,\ldots,0\right)  =1$, as $g\in\mathcal{K}\subseteq W^{2}$. We
also have $g\left(  1,\ldots,1,1,\ldots,1\right)  =1$, and then we must have
$g\left(  0,\ldots,0,0,\ldots,0\right)  =1$, since otherwise
$\operatorname{id}$ would be a subfunction of $g$, implying $\mathcal{K}%
=\left[  \mathcal{K}\right]  =W^{k}$, contrary to our assumption. Therefore
$g\left(  x_{i},\ldots,x_{i},y_{j},\ldots,y_{j}\right)  =0$ iff $x_{i}=0$ and
$y_{j}=1$, hence the disjunction in the definition of $f^{\prime}$ equals $0$
iff $x_{1}=\cdots=x_{l}=0$ and $y_{1}=\cdots=y_{n-l}=1$. Thus $f^{\prime}$
differs from $g$ only at the position $\mathbf{a}$, where $f^{\prime}\left(
\mathbf{a}\right)  =1$ and $g\left(  \mathbf{a}\right)  =0$. Hence
$f=f^{\prime}\in W^{\infty}\circ\mathcal{K}\subseteq\mathcal{K}$, as claimed.
\end{proof}

\begin{proposition}
\label{prop W_infty}Let $\mathcal{K}$ be an equational class such that
$\left[  \mathcal{K}\right]  =W^{\infty}$. Then%
\[
\mathcal{K}\in\mathbf{I}\left(  W^{\infty}\right)  \iff Z_{\infty}%
\mathcal{K}=\mathcal{K}.
\]

\end{proposition}

\begin{proof}
\textquotedblleft$\implies$\textquotedblright: Follows from the previous lemma
with $k=\infty$.

\textquotedblleft$\impliedby$\textquotedblright: Let us suppose that $\left[
\mathcal{K}\right]  =W^{\infty}$ and $Z_{\infty}\mathcal{K}=\mathcal{K}$. Let
$f\in\mathcal{K}$ be $n$-ary and $g_{1},\ldots,g_{n}\in\mathcal{K}$ be $k$-ary
functions. We have to prove that $h=f\left(  g_{1},\ldots,g_{n}\right)  $
belongs to $\mathcal{K}$. Since $f\in W^{\infty}$, there exists an index $i$
such that for all $\left(  a_{1},\ldots,a_{n}\right)  \in f^{-1}\left(
0\right)  $ we have $a_{i}=0$. Therefore, for all $\left(  b_{1},\ldots
,b_{k}\right)  \in h^{-1}\left(  0\right)  $ we have $g_{i}\left(
b_{1},\ldots,b_{k}\right)  =0$, i.e., $h^{-1}\left(  0\right)  \subseteq
g_{i}^{-1}\left(  0\right)  $. This proves that $h\in Z_{\infty}%
\mathcal{K}=\mathcal{K}$.
\end{proof}

\begin{remark}
\label{rem ideal}Let us note that $Z_{\infty}\mathcal{K}=\mathcal{K}$ iff
$\mathcal{K}$ is a filter in the usual pointwise ordering $\leq$ of Boolean
functions. Let $\sqsubseteq$ be the transitive closure of $\preceq\cup\geq$.
Then $\sqsubseteq$ is a quasiorder on $\Omega$, and a class $\mathcal{K}$ is
an ideal with respect to this quasiorder iff $\mathcal{K}$ is an ideal w.r.t.
$\preceq$ (i.e., an equational class) and a filter w.r.t. $\leq$ (i.e.,
satisfies $Z_{\infty}\mathcal{K}=\mathcal{K}$). Thus we can reformulate the
above proposition as follows: $\mathcal{K}\in\mathbf{I}\left(  W^{\infty
}\right)  $ iff $\left[  \mathcal{K}\right]  =W^{\infty}$ and $\mathcal{K}$ is
an ideal w.r.t. $\sqsubseteq$. This implies the somewhat surprising fact that
a union of idempotents is idempotent in this case, i.e., the lattice
operations in the interval $\mathbf{I}\left(  W^{\infty}\right)  $ coincide
with the set operations $\cap$ and $\cup$. Consequently, $\mathbf{I}\left(
W^{\infty}\right)  $ is a distributive lattice.
\end{remark}

In the next proposition we describe $\bigcap\mathbf{I}\left(  W^{\infty
}\right)  $, the bottom element of the interval $\mathbf{I}\left(  W^{\infty
}\right)  $.

\begin{proposition}
\label{prop bottom W_infty}The bottom element of the interval $\mathbf{I}%
\left(  W^{\infty}\right)  $ is $\left\lfloor \rightarrow\right\rfloor $,
which is strongly defined by the relational constraints%
\[
\Bigl(\left\{  0,1\right\}  ^{k}\setminus\left\{  \mathbf{1}\right\}
,\left\{  0,1\right\}  ^{k}\setminus\left\{  \mathbf{0}\right\}
\Bigr)\quad\left(  k\in\mathbb{N}\right)  .
\]

\end{proposition}

\begin{proof}
Let $\mathcal{B}^{\infty}$ denote the function class strongly defined by the
above constraints. We will prove that%
\[
\bigcap\mathbf{I}\left(  W^{\infty}\right)  \subseteq\left\lfloor
\rightarrow\right\rfloor \subseteq\mathcal{B}^{\infty}\subseteq\bigcap
\mathbf{I}\left(  W^{\infty}\right)  .
\]

The first containment follows immediately from the fact that $\left\lfloor
\rightarrow\right\rfloor \in\mathbf{I}\left(  W^{\infty}\right)  $, and for
the second one it suffices to verify that $\rightarrow\in\mathcal{B}^{\infty}%
$, since $\mathcal{B}^{\infty}$ is an idempotent according to
Theorem~\ref{thm galois connection}.

For the third containment, let $f$ be an arbitrary $n$-ary function in
$\mathcal{B}^{\infty}$, and let $k=\left\vert f^{-1}\left(  0\right)
\right\vert $. Since $f$ preserves the relation $\left\{  0,1\right\}
^{k}\setminus\left\{  \mathbf{0}\right\}  $, the matrix $f^{-1}\left(
0\right)  $ has a constant $0$ column. Moreover, since $f$ satisfies the
relational constraint $(\left\{  0,1\right\}  ^{k}\setminus\left\{
\mathbf{1}\right\}  ,\left\{  0,1\right\}  ^{k}\setminus\left\{
\mathbf{0}\right\}  )$, the matrix $f^{-1}\left(  0\right)  $ has a constant
$1$ column as well. We can suppose without loss of generality that the first
column of $f^{-1}\left(  0\right)  $ is constant $1$ and the second column is
constant $0$.

Since $\left[  \bigcap\mathbf{I}\left(  W^{\infty}\right)  \right]
=W^{\infty}$, there is at least one nonmonotone function $g_{1}\in
\bigcap\mathbf{I}\left(  W^{\infty}\right)  $. For such a function,
$g_{1}^{-1}\left(  0\right)  \nsubseteq\left\{  \mathbf{0},\mathbf{1}\right\}
$, thus, permuting variables, we may suppose that $g_{1}^{-1}\left(  0\right)
$ contains a vector of the form $\left(  1,\ldots,1,0,\ldots,0\right)  $.
Applying the operator $Z_{\infty}$, we can obtain a function $g_{2}\in
Z_{\infty}\bigcap\mathbf{I}\left(  W^{\infty}\right)  =\bigcap\mathbf{I}%
\left(  W^{\infty}\right)  $ of the same arity as $g_{1}$ with $g_{2}%
^{-1}\left(  0\right)  =\left\{  \left(  1,\ldots,1,0,\ldots,0\right)
\right\}  $. Identifying appropriately the variables of $g_{2}$, we get a
binary function $g_{3}\in\bigcap\mathbf{I}\left(  W^{\infty}\right)  $ with
$g_{3}^{-1}\left(  0\right)  =\left\{  \left(  1,0\right)  \right\}  $
($g_{3}$ is nothing else but the implication). Adding $n-2$ dummy variables,
we can obtain a function $g_{4}\in\bigcap\mathbf{I}\left(  W^{\infty}\right)
$ of arity $n$, such that $g_{4}^{-1}\left(  0\right)  =\left\{  1\right\}
\times\left\{  0\right\}  \times\left\{  0,1\right\}  ^{n-2}$, i.e.,
\emph{every} $n$-tuple of the form $\left(  1,0,\ldots\right)  $ belongs to
$g_{4}^{-1}\left(  0\right)  $. Since $f^{-1}\left(  0\right)  $ consists of
\emph{some} of these tuples, we have $f\in Z_{\infty}\left\{  g_{4}\right\}
\subseteq Z_{\infty}\bigcap\mathbf{I}\left(  W^{\infty}\right)  =\bigcap
\mathbf{I}\left(  W^{\infty}\right)  $.
\end{proof}

Combining the previous two propositions and Remark~\ref{rem ideal}, we get the
following characterizations of the idempotents corresponding to $W^{\infty}$.

\begin{theorem}
\label{thm W_infty}For any class $\mathcal{K}$ of Boolean functions the
following conditions are equivalent:

\begin{enumerate}
\item $\mathcal{K}\in\mathbf{I}\left(  W^{\infty}\right)  $;

\item $\rightarrow\in\mathcal{K}\subseteq W^{\infty}$, and $\mathcal{K}$ is an
equational class satisfying $Z_{\infty}\mathcal{K}=\mathcal{K}$;

\item $\rightarrow\in\mathcal{K}\subseteq W^{\infty}$, and $\mathcal{K}$ is an
ideal with respect to the quasiorder $\sqsubseteq$.
\end{enumerate}
\end{theorem}

We conclude this section by proving that the interval $\mathbf{I}\left(
W^{\infty}\right)  $ is uncountable.

\begin{theorem}
The interval $\mathbf{I}\left(  W^{\infty}\right)  $ has continuum cardinality.
\end{theorem}

\begin{proof}
Let $J_{n}$ be the following $n\times\left(  n+1\right)  $ matrix over
$\left\{  0,1\right\}  $:%
\[
J_{n}=%
\begin{pmatrix}
1 & 1 & 0 & 0 & \cdots & 0 & 0 & 0 & 0\\
0 & 1 & 1 & 0 & \cdots & 0 & 0 & 0 & 0\\
0 & 0 & 1 & 1 & \cdots & 0 & 0 & 0 & 0\\
0 & 0 & 0 & 1 & \cdots & 0 & 0 & 0 & 0\\
\vdots & \vdots & \vdots & \vdots & \ddots & \vdots & \vdots & \vdots &
\vdots\\
0 & 0 & 0 & 0 & \cdots & 1 & 1 & 0 & 0\\
0 & 0 & 0 & 0 & \cdots & 0 & 1 & 1 & 0\\
1 & 0 & 0 & 0 & \cdots & 0 & 0 & 1 & 0
\end{pmatrix}
,
\]
and let $f_{n}$ be the $\left(  n+1\right)  $-ary Boolean function such that
$f_{n}^{-1}\left(  0\right)  $ consists of the rows of $J_{n}$, and let
$P_{n}$ be the $n$-ary relation consisting of the columns of $J_{n}$. It is
straightforward to verify that $\left[  f_{n}\right]  =W^{\infty}$, hence
$\left\lfloor f_{n}\right\rfloor \in\mathbf{I}\left(  W^{\infty}\right)  $. We
claim that for all natural numbers $m,n$,%
\[
\text{if }m\text{ is odd, then }f_{n}\text{ strongly satisfies }\left(
P_{m},\left\{  0,1\right\}  ^{m}\setminus\left\{  \mathbf{0}\right\}  \right)
\text{ iff }m\neq n.
\]

It is clear that $f_{n}$ does not satisfy $\left(  P_{n},\left\{  0,1\right\}
^{n}\setminus\left\{  \mathbf{0}\right\}  \right)  $, since $f_{n}\left(
J_{n}\right)  =\mathbf{0}$. Now let us assume that $m$ is odd and $m\neq n$.
The matrix $J_{n}$ has a constant $0$ column, therefore $f_{n}$ preserves
$\left\{  0,1\right\}  ^{m}\setminus\left\{  \mathbf{0}\right\}  $. Suppose
for contradiction that $f_{n}$ does not satisfy $\left(  P_{m},\left\{
0,1\right\}  ^{m}\setminus\left\{  \mathbf{0}\right\}  \right)  $, i.e., there
exists a $P_{m}$-matrix $N$ of size $m\times\left(  n+1\right)  $ such that
$f\left(  N\right)  =\mathbf{0}$. This means that every column of $N$ is a
column of $J_{m}$, and every row of $N$ is a row of $J_{n}$.

Let us interpret the matrix $N$ as the incidence matrix of a graph $G$ (with
possibly multiple edges): the columns of $N$ correspond to the vertices
$v_{1},\ldots,v_{n+1}$ of $G$, and the rows of $N$ correspond to the $m$ edges
of $G$. Since each row of $N$ is a row of $J_{n}$, the vertex $v_{n+1}$ is
isolated, and every edge connects two consecutive vertices in the cyclical
ordering of $v_{1},\ldots,v_{n}$. Since every column of $N$ is a column of
$J_{m}$, every vertex has degree $0$ or $2$. These properties imply that $G$
is either a cycle on the vertices $v_{1},\ldots,v_{n}$ (together with the
isolated vertex $v_{n+1}$), or the components of $G$ are $2$-cycles (double
edges) and isolated vertices:

\medskip

\hspace{\stretch{0.1999}}%
{\includegraphics[
height=1.4236in,
width=1.4369in
]%
{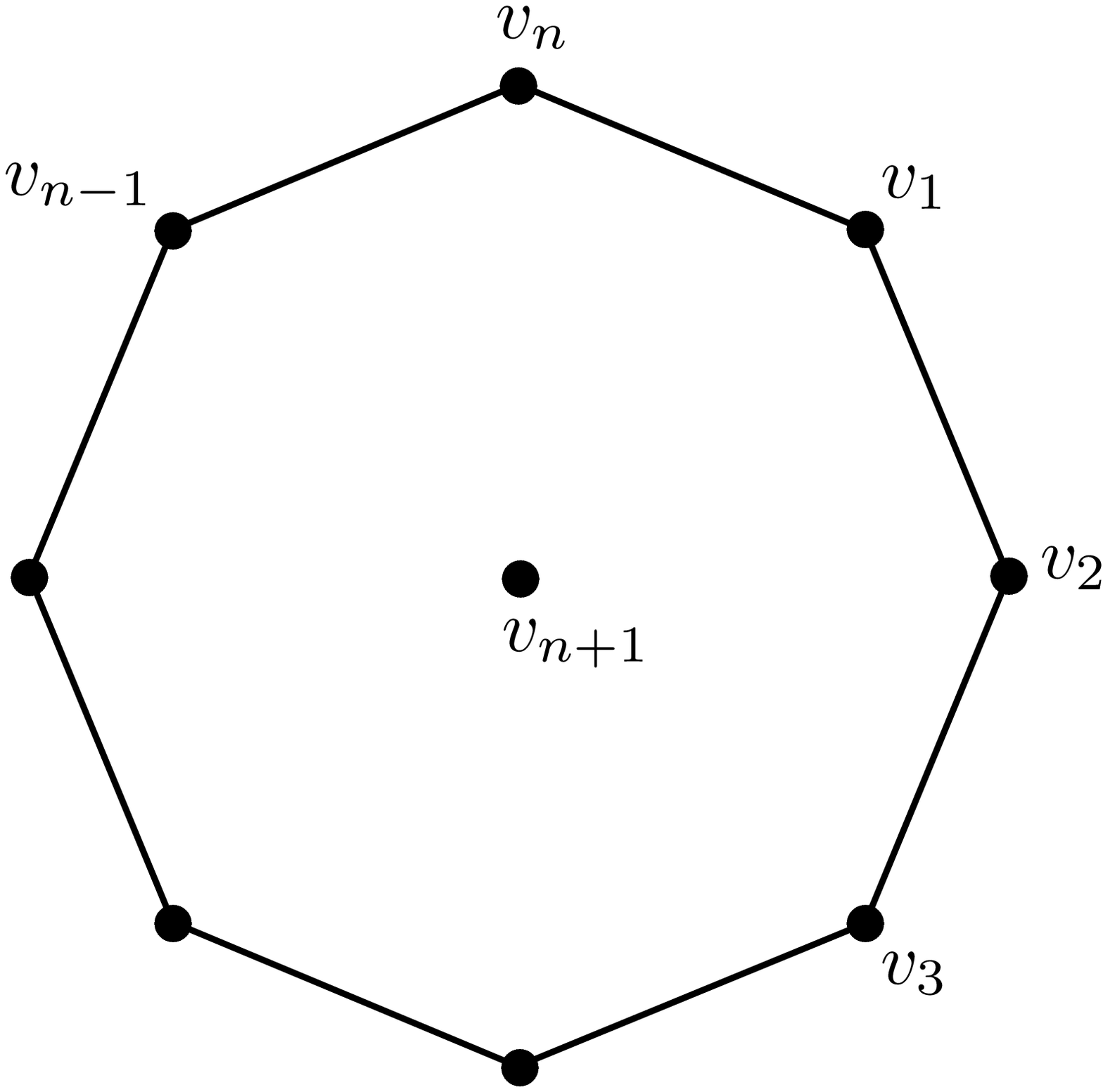}%
}
\hfill%
{\includegraphics[
height=1.4236in,
width=1.4369in
]%
{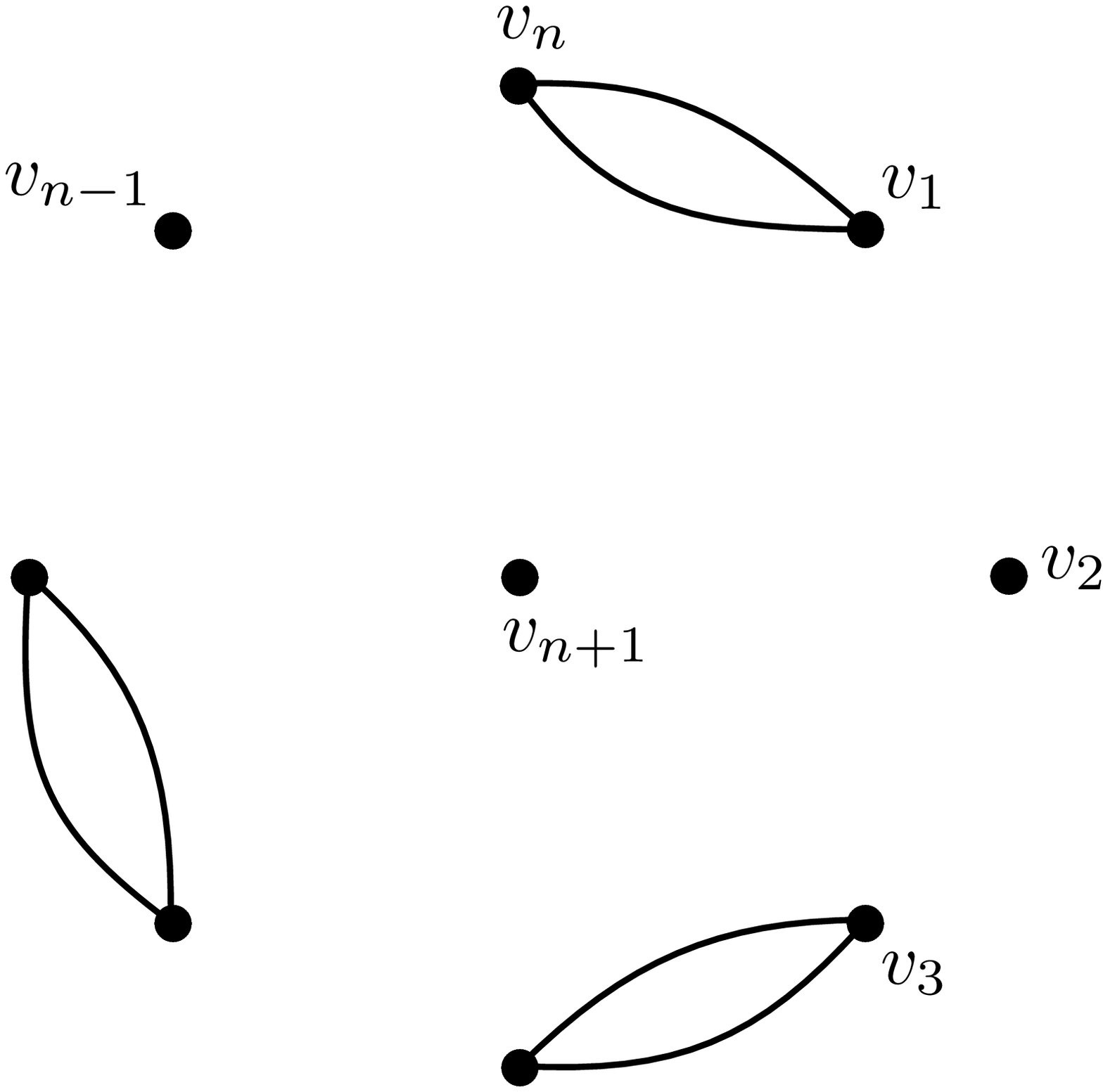}%
}
\hspace{\stretch{0.1999}}

\medskip

\noindent Both cases lead to a contradiction: in the former case we have
$m=n$, while in the latter case we can conclude that $m$ is even. This
contradiction proves our claim.

To finish the proof of the theorem it suffices to observe that our claim
implies that $I\mapsto\left\lfloor \left\{  f_{i}:i\in I\right\}
\right\rfloor $ embeds the power set of the set of odd natural numbers into
$\mathbf{I}\left(  W^{\infty}\right)  $.
\end{proof}

\begin{remark}
We have shown that the power set of a countably infinite set embeds into
$\mathbf{I}\left(  W^{\infty}\right)  $, and it is obvious that $\mathbf{I}%
\left(  W^{\infty}\right)  $ embeds into the power set of $\Omega$. Thus,
$\mathbf{I}\left(  W^{\infty}\right)  $ is equimorphic to the power set of a
countably infinite set.
\end{remark}

\pagebreak

\section{Idempotents corresponding to $W^{k}$\label{sect W_k}}

We will prove that the analogue of Proposition~\ref{prop W_infty} holds for
finite values of $k$ as well. First we need a technical lemma.

\begin{lemma}
For any $\mathcal{K}\subseteq\Omega$ and $l\geq2$ we have $Z_{l}%
\mathcal{K}\subseteq Z_{l+1}\left(  W^{l}\circ\mathcal{K}\right)  $.
\end{lemma}

\begin{proof}
Let $f$ be an $n$-ary function in $Z_{l}\mathcal{K}$, and let $\mathbf{a}%
_{1},\ldots,\mathbf{a}_{l+1}\in f^{-1}\left(  0\right)  $. We need to show
that there exists an $n$-ary function $u\in W^{l}\circ\mathcal{K}$ such that
$\mathbf{a}_{1},\ldots,\mathbf{a}_{l+1}\in u^{-1}\left(  0\right)  $. Since
$f\in Z_{l}\mathcal{K}$, for every $j=1,\ldots,l+1$ we can find a function
$g_{j}\in\mathcal{K}$ such that $\mathbf{a}_{1},\ldots,\mathbf{a}%
_{j-1},\mathbf{a}_{j+1},\ldots,\mathbf{a}_{l+1}\in g_{j}^{-1}\left(  0\right)
$. Now let us consider the following function $h$ of arity $l+1$:%
\[
h\left(  x_{1},\ldots,x_{l+1}\right)  =\left\{  \!\!%
\begin{array}
[c]{ll}%
0, & \text{if }\left\vert \left\{  i:x_{i}=1\right\}  \right\vert \leq1;\\
1, & \text{otherwise.}%
\end{array}
\right.
\]
It is easy to see that $h\in W^{l}$, therefore $u=h\left(  g_{1}%
,\ldots,g_{l+1}\right)  \in W^{l}\circ\mathcal{K}$, and $\mathbf{a}_{1}%
,\ldots,\mathbf{a}_{l+1}\in u^{-1}\left(  0\right)  $. This proves that $f\in
Z_{l+1}\left(  W^{l}\circ\mathcal{K}\right)  $.
\end{proof}

\begin{proposition}
\label{prop W_k}Let $\mathcal{K}$ be an equational class such that $\left[
\mathcal{K}\right]  =W^{k}$. Then%
\[
\mathcal{K}\in\mathbf{I}\left(  W^{k}\right)  \iff Z_{k}\mathcal{K}%
=\mathcal{K}.
\]

\end{proposition}

\begin{proof}
\textquotedblleft$\implies$\textquotedblright: Let $\mathcal{K}$ be an
idempotent such that $\left[  \mathcal{K}\right]  =W^{k}$.
Lemma~\ref{lemma ize} shows that $Z_{\infty}\mathcal{K}=\mathcal{K}$. For any
$l\geq k$ we have $W^{l}\circ\mathcal{K}\subseteq W^{k}\circ\mathcal{K}%
=\mathcal{K}$ by Proposition~\ref{prop CK}, and using the previous lemma we
get $Z_{l}\mathcal{K}\subseteq Z_{l+1}\left(  W^{l}\circ\mathcal{K}\right)
\subseteq Z_{l+1}\mathcal{K}$. The reversed containment $Z_{l}\mathcal{K}%
\supseteq Z_{l+1}\mathcal{K}$ is obvious, so we have $Z_{k}\mathcal{K}%
=Z_{k+1}\mathcal{K}=Z_{k+2}\mathcal{K}=\cdots=Z_{\infty}\mathcal{K}%
=\mathcal{K}$ in light of (\ref{eq Z_2...Z_infty}).

\textquotedblleft$\impliedby$\textquotedblright: We just need to modify
slightly the second part of the proof of Proposition~\ref{prop W_infty}. Let
us suppose that $\mathcal{K}$ is an equational class such that $\left[
\mathcal{K}\right]  =W^{k}$ and $\mathcal{K}$ is closed under the operator
$Z_{k}$. Let $f\in\mathcal{K}$ be $n$-ary, and $g_{1},\ldots,g_{n}%
\in\mathcal{K}$ be $m$-ary functions. We have to prove that $h=f\left(
g_{1},\ldots,g_{n}\right)  $ belongs to $\mathcal{K}=Z_{k}\mathcal{K}$. If
$\mathbf{a}_{1},\ldots,\mathbf{a}_{k}\in h^{-1}\left(  0\right)  $, then the
vectors $\left(  g_{1}\left(  \mathbf{a}_{j}\right)  ,\ldots,g_{n}\left(
\mathbf{a}_{j}\right)  \right)  $ belong to $f^{-1}\left(  0\right)  $ for
$j=1,\ldots,k$. Since $f\in W^{k}$, there exists an index $i$ such that
$g_{i}\left(  \mathbf{a}_{j}\right)  =0$ for $j=1,\ldots,k$, i.e.,
$\mathbf{a}_{1},\ldots,\mathbf{a}_{k}\in g_{i}^{-1}\left(  0\right)  $. This
shows that $h\in Z_{k}\mathcal{K}=\mathcal{K}$.
\end{proof}

Our next task is, just like in the previous section, to describe
$\bigcap\mathbf{I}\left(  W^{k}\right)  $, the bottom element of the interval
$\mathbf{I}\left(  W^{k}\right)  $. The proof is very similar to the proof of
Proposition~\ref{prop bottom W_infty}.

\begin{proposition}
\label{prop bottom W_k}The bottom element of the interval $\mathbf{I}\left(
W^{k}\right)  $ is $\left\lfloor w_{k}\right\rfloor $, which is strongly
defined by the relational constraint%
\[
\Bigl(\left\{  0,1\right\}  ^{k}\setminus\left\{  \mathbf{1}\right\}
,\left\{  0,1\right\}  ^{k}\setminus\left\{  \mathbf{0}\right\}  \Bigr).
\]

\end{proposition}

\begin{proof}
Let $\mathcal{B}^{k}$ denote the function class strongly defined by the above
constraint. We will prove that%
\[
\bigcap\mathbf{I}\left(  W^{k}\right)  \subseteq\left\lfloor w_{k}%
\right\rfloor \subseteq\mathcal{B}^{k}\subseteq\bigcap\mathbf{I}\left(
W^{k}\right)  .
\]

The first containment follows immediately from the fact that $\left\lfloor
w_{k}\right\rfloor \in\mathbf{I}\left(  W^{k}\right)  $, and for the second
one it suffices to verify that $w_{k}\in\mathcal{B}^{k}$, since $\mathcal{B}%
^{k}$ is an idempotent according to Theorem~\ref{thm galois connection}.

For the third containment, let $f$ be an arbitrary $n$-ary function in
$\mathcal{B}^{k}$, and $N$ be an at most $k$-element subset of $f^{-1}\left(
0\right)  $. Since $f$ preserves the relation $\left\{  0,1\right\}
^{k}\setminus\left\{  \mathbf{0}\right\}  $, viewing $N$ as a matrix, it has a
constant $0$ column. Moreover, since $f$ satisfies the constraint $(\left\{
0,1\right\}  ^{k}\setminus\left\{  \mathbf{1}\right\}  ,\left\{  0,1\right\}
^{k}\setminus\left\{  \mathbf{0}\right\}  )$, the matrix $N$ has a constant
$1$ column as well. We can suppose without loss of generality that the first
column of $N$ is constant $1$ and the second column is constant $0$.

Since $\left[  \bigcap\mathbf{I}\left(  W^{k}\right)  \right]  =W^{k}$, there
is at least one nonmonotone function in $\bigcap\mathbf{I}\left(
W^{k}\right)  $, and taking into account that $\bigcap\mathbf{I}\left(
W^{k}\right)  $ is closed under the operator $Z_{\infty}$, we can apply the
same argument as in the proof of Proposition~\ref{prop bottom W_infty} to
construct a function $g_{4}\in\bigcap\mathbf{I}\left(  W^{k}\right)  $ of
arity $n$, such that $g_{4}^{-1}\left(  0\right)  =\left\{  1\right\}
\times\left\{  0\right\}  \times\left\{  0,1\right\}  ^{n-2}$. Then we have
$N\subseteq g_{4}^{-1}\left(  0\right)  $, and since we can construct such a
function $g_{4}$ for any at most $k$-element subset $N$ of $f^{-1}\left(
0\right)  $, we can conclude that $f\in Z_{k}\bigcap\mathbf{I}\left(
W^{k}\right)  =\bigcap\mathbf{I}\left(  W^{k}\right)  $.
\end{proof}

The previous two propositions yield the following characterization of the
idempotents corresponding to $W^{k}$.

\begin{theorem}
\label{thm W_k}For any class $\mathcal{K}$ of Boolean functions the following
conditions are equivalent:

\begin{enumerate}
\item $\mathcal{K}\in\mathbf{I}\left(  W^{k}\right)  $;

\item $w_{k}\in\mathcal{K}\subseteq W^{k}$, and $\mathcal{K}$ is an equational
class satisfying $Z_{k}\mathcal{K}=\mathcal{K}$.
\end{enumerate}
\end{theorem}

Finally, we prove that the intervals $\mathbf{I}\left(  W^{k}\right)  $ are
finite, but their sizes do not have a common upper bound.

\begin{theorem}
For any $k\geq2$, the interval $\mathbf{I}\left(  W^{k}\right)  $ is finite
and has at least $k+1$ elements.
\end{theorem}

\begin{proof}
To obtain the lower bound, let us observe that for any $j,k\geq2$ we have
$\mathcal{B}^{j}\cap W^{k}\supseteq\mathcal{B}^{j+1}\cap W^{k}$, and this
containment is proper, since $v_{j}\in\mathcal{B}^{j}\cap W^{k}$ and
$v_{j}\notin\mathcal{B}^{j+1}\cap W^{k}$, where $v_{j}\left(  x_{1}%
,\ldots,x_{j+2}\right)  =w_{j}\left(  \lnot x_{1},\ldots,\lnot x_{j+2}\right)
$. Thus we have the following chain of length $k+1$ in $\mathbf{I}\left(
W^{k}\right)  $ (see Figure~\ref{fig intervals}):%
\[
W^{k}\supset W_{=}^{k}\supset\mathcal{B}^{2}\cap W^{k}\supset\mathcal{B}%
^{3}\cap W^{k}\supset\cdots\supset\mathcal{B}^{k}\cap W^{k}=\mathcal{B}^{k}.
\]%
\begin{figure}
[ptb]
\begin{center}
\includegraphics[
height=7.9303cm,
width=9.8386cm
]%
{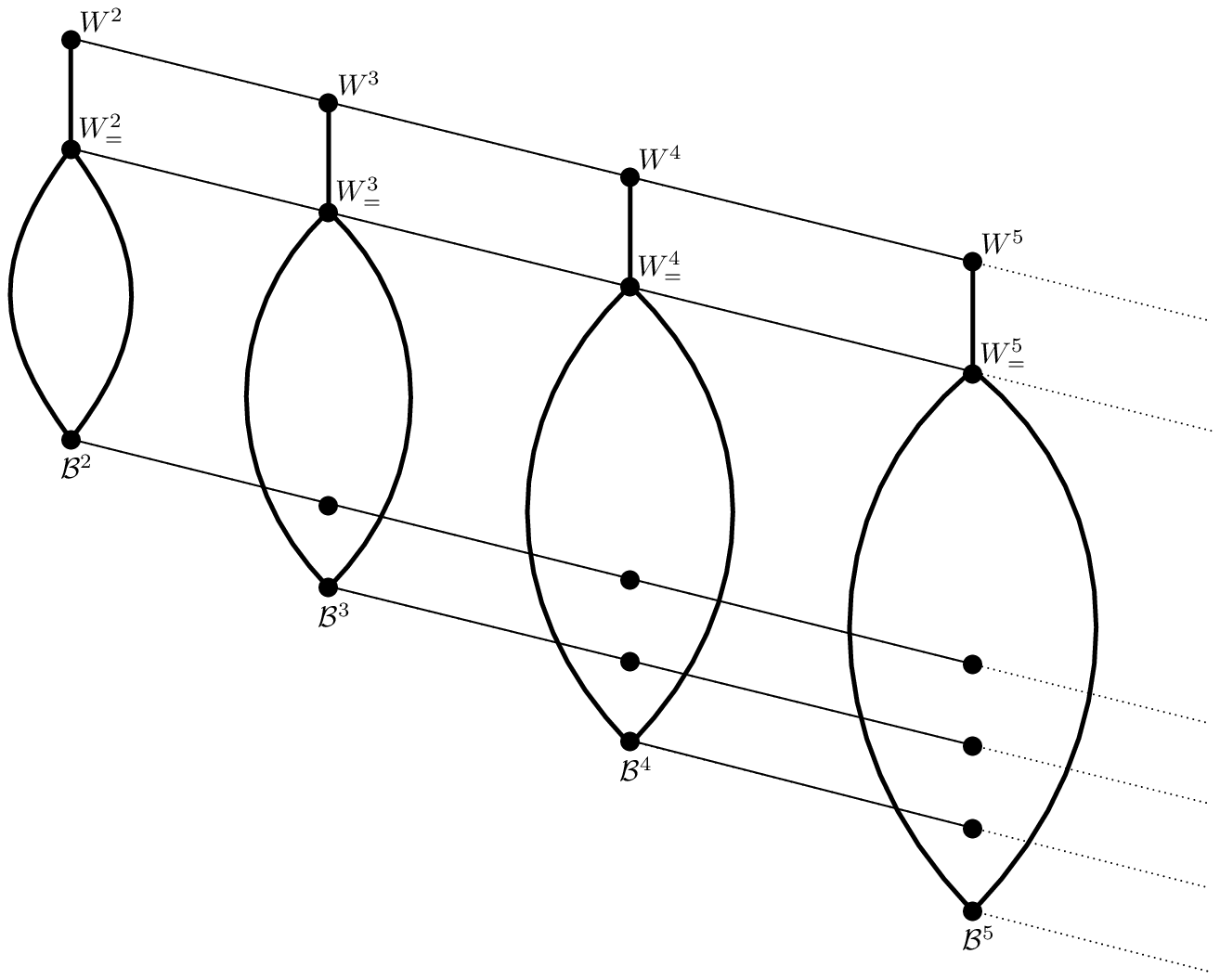}%
\caption{The intervals $\mathbf{I}\left(  W^{k}\right)  $}%
\label{fig intervals}%
\end{center}
\end{figure}

In order to prove the finiteness of $\mathbf{I}\left(  W^{k}\right)  $, we
define the \emph{skeleton} of an idempotent class $\mathcal{K}\in
\mathbf{I}\left(  W^{k}\right)  $ as follows. For every $f\in\mathcal{K}$, let
us construct all matrices formed by at most $k$ rows of the matrix
$f^{-1}\left(  0\right)  $, and delete repeated columns, if there are any. The
skeleton of $\mathcal{K}$ is the collection of all such matrices. Every matrix
in the skeleton has at most $k$ rows, and at most $2^{k}$ columns, since there
are no repeated columns. There are only finitely many such matrices, hence
there are only finitely many possible skeletons. Therefore, it suffices to
prove that different idempotents in $\mathbf{I}\left(  W^{k}\right)  $ have
different skeletons.

So let us suppose that $\mathcal{K}_{1},\mathcal{K}_{2}\in\mathbf{I}\left(
W^{k}\right)  $ have the same skeleton $\mathcal{S}$. Let $f_{1}$ be any
function in $\mathcal{K}_{1}$, and let $N_{1}\subseteq f_{1}^{-1}\left(
0\right)  $ be any set with at most $k$ elements. Deleting repeated columns of
the matrix $N_{1}$, we obtain a matrix $N^{\prime}\in\mathcal{S}$. Since
$\mathcal{S}$ is the skeleton of $\mathcal{K}_{2}$ as well, there exists a
function $f_{2}\in\mathcal{K}_{2}$ and a matrix $N_{2}$ formed by at most $k$
rows of $f_{2}^{-1}\left(  0\right)  $, such that deleting repeated columns of
$N_{2}$ we obtain the same matrix $N^{\prime}$. Identifying variables of
$f_{2}$ we can construct a function $g_{2}\in\mathcal{K}_{2}$ such that
$g_{2}^{-1}\left(  0\right)  \supseteq N^{\prime}$. Now adding dummy variables
to $g_{2}$, we can construct a function $h_{2}\in\mathcal{K}_{2}$ such that
$h_{2}^{-1}\left(  0\right)  \supseteq N_{1}$. Since we can do this for any at
most $k$-element subset $N_{1}$ of $f_{1}^{-1}\left(  0\right)  $, we can
conclude that $f_{1}\in Z_{k}\mathcal{K}_{2}=\mathcal{K}_{2}$. This proves
that $\mathcal{K}_{1}\subseteq\mathcal{K}_{2}$, and a similar argument yields
$\mathcal{K}_{1}\supseteq\mathcal{K}_{2}$, thus we have $\mathcal{K}%
_{1}=\mathcal{K}_{2}$.
\end{proof}

\section{Concluding remarks\label{sect concluding}}

Assembling the results of the previous sections, we can draw the lattice of
projection-free idempotents as shown in Figure~\ref{fig lattice}. The bottom
of the interval $\mathbf{I}\left(  U^{k}\right)  $ is denoted by
$\mathcal{D}^{k}$; it is the dual of $\mathcal{B}^{k}$, hence it can be
strongly defined by the relational constraint $(\left\{  0,1\right\}
^{k}\setminus\left\{  \mathbf{0}\right\}  ,\left\{  0,1\right\}  ^{k}%
\setminus\left\{  \mathbf{1}\right\}  )$ for finite $k$, and by all of these
constraints for $k=\infty$. To obtain the whole lattice $\mathbf{I}$, we have
to put together this lattice with the Post lattice, as shown schematically in
Figure~\ref{fig biglattice}. For the \textquotedblleft real
picture\textquotedblright\ we would have to connect $\mathcal{C}_{=}$ to
$\mathcal{C}$ for each nontrivial interval $\mathbf{I}\left(  \mathcal{C}%
\right)  $, but this would make the figure incomprehensibly complex.%

\begin{figure}
[t]
\begin{center}
\includegraphics[
height=8.2728cm,
width=11.8437cm
]%
{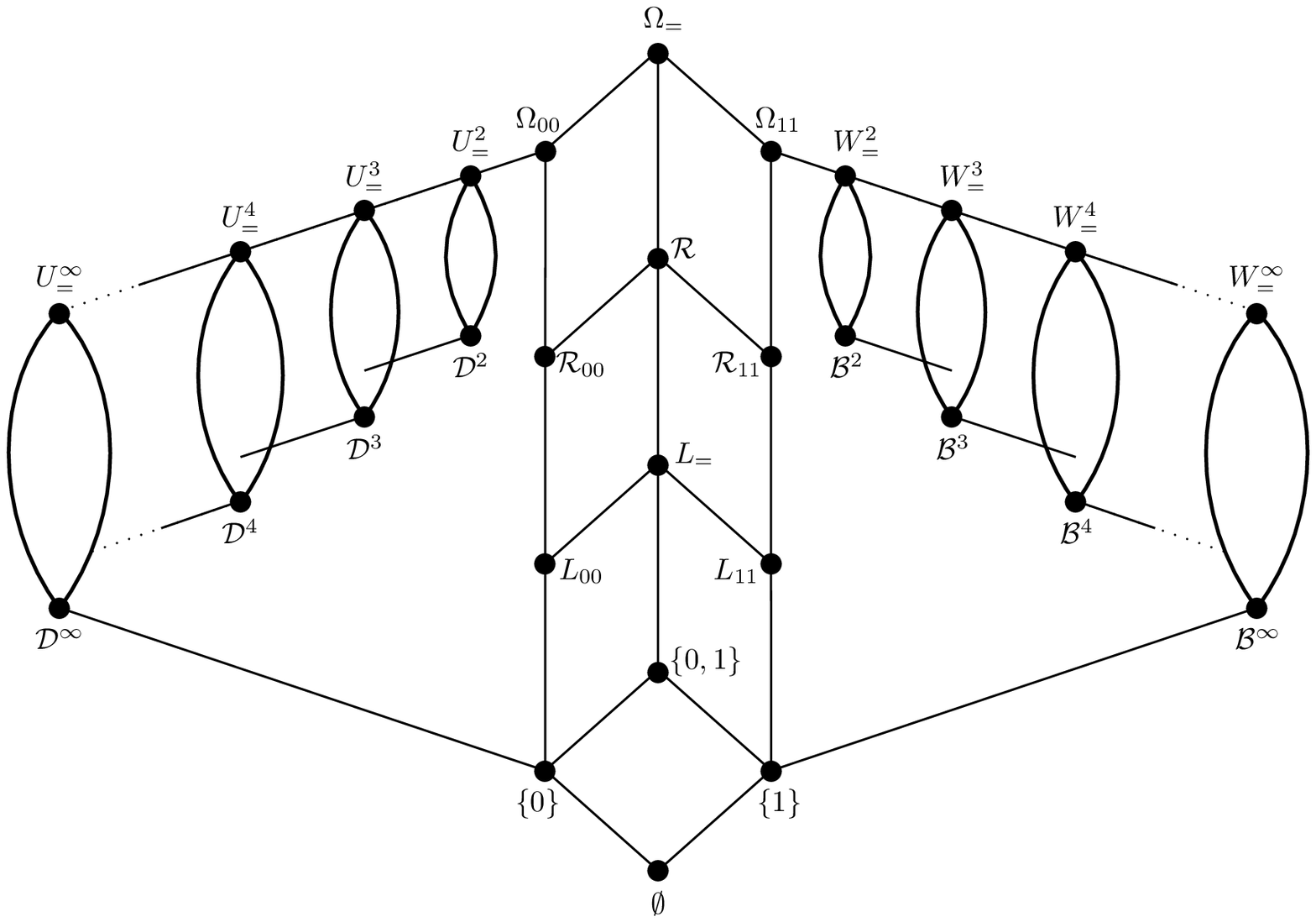}%
\caption{The lattice of closed classes without projections}%
\label{fig lattice}%
\end{center}
\end{figure}

\begin{figure}
[t]
\begin{center}
\includegraphics[
height=8.3584cm,
width=11.8414cm
]%
{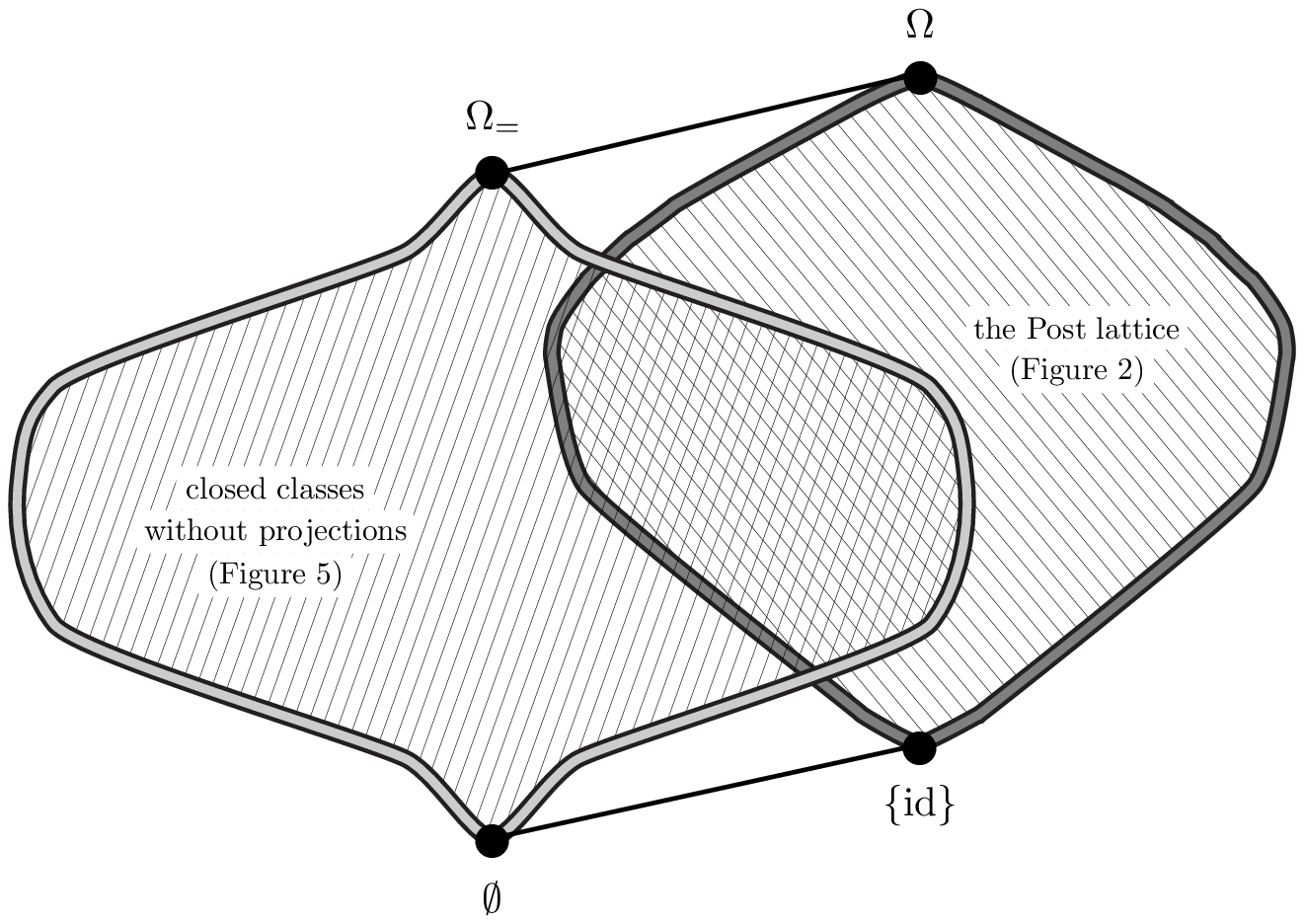}%
\caption{The lattice $\mathbf{I}$ of composition-closed equational classes of
Boolean functions}%
\label{fig biglattice}%
\end{center}
\end{figure}

Finally, let us mention a few directions for further investigations. Our
characterization of the intervals $\mathbf{I}\left(  W^{\infty}\right)  $ and
$\mathbf{I}\left(  W^{k}\right)  $ is not explicit, hence a more concrete
description would be desirable. In particular, it would be interesting to
determine (at least asymptotically) the size of $\mathbf{I}\left(
W^{k}\right)  $. To better understand the structure of $\mathbf{I}\left(
W^{\infty}\right)  $, the quasiorder $\sqsubseteq$ defined in
Remark~\ref{rem ideal} should be studied. Although the lattice of clones over
a base set with at least three elements is not fully described, it may be
possible to get some results about composition-closed equational classes over
arbitrary finite domains, e.g., determine minimal and maximal closed classes.
The description of $\mathbf{I}$ obtained in this paper can be regarded as a
first step in the study of the semigroup $\left(  \mathbf{E};\circ\right)  $;
for further results in this direction see \cite{ACW}.

\section*{Acknowledgments}

The present project is supported by the National Research Fund, Luxembourg,
and cofunded under the Marie Curie Actions of the European Commission
(FP7-COFUND), and supported by the Hungarian National Foundation for
Scientific Research under grants no. {T48809, K60148 and }K77409. The author
would like express his gratitude to Miguel Couceiro, P\'{e}ter Hajnal, Erkko
Lehtonen and the anonymous referee for helpful comments and suggestions.

\end{document}